\documentclass[11pt]{article}
\usepackage{amsmath}
\usepackage{amsfonts}
\usepackage{amssymb}
\usepackage{amsthm}
\usepackage{graphicx}
\usepackage{psfrag}
\newtheorem{theorem}{Theorem}[section]
\newtheorem{proposition}[theorem]{Proposition}
\newtheorem{definition}[theorem]{Definition}
\newtheorem{lemma}[theorem]{Lemma}
\newtheorem{corollary}[theorem]{Corollary}

\newtheorem{assumption}[theorem]{Assumption}
\theoremstyle{remark}
\newtheorem{example}{Example}

\newcommand{\fraku}{\mathfrak{U}}

\newcommand{\spam}{\mathop{\mathrm{span}}}

\newcommand{\ints}{\mathbb{Z}}
\newcommand{\nats}{\mathbb{N}}
\newcommand{\reals}{\mathbb{R}}
\newcommand{\RR}{\mathbb{R}}
\newcommand{\comps}{\mathbb{C}}
\newcommand{\M}{\mathbb{M}}
\newcommand{\calo}{\mathcal{O}}

\newcommand{\sphere}{\mathbb{S}^d}
\newcommand{\sph}{\mathbb{S}} 
\newcommand{\Exp}{\operatorname{Exp}}
\newcommand{\man}{\M}
\newcommand{\dif}{\mathrm{d}}

\renewcommand{\b}{\mathbf{b}}
\renewcommand{\a}{\mathbf{a}}
\newcommand{\bft}{\mathbf{t}}
\newcommand{\bfu}{\mathbf{u}}
\newcommand{\bfv}{\mathbf{v}}
\newcommand{\bfw}{\mathbf{w}}

\newcommand{\inj}{\mathrm{r}_\M}

\renewcommand{\d}{\mathrm{dist}}
\renewcommand{\aa}{\mathfrak{a}}

\newcommand{\bb}{\mathfrak{b}}

\newcommand{\D}{\{1...d\}}
\newcommand{\sfp}{\mathsf{p}}
\newcommand{\sfq}{\mathsf{q}}
\newcommand{\stardom}{\mathcal D}  
\newcommand{\inrad}{r}  
\newcommand{\J}{\mathcal{J}}
\newcommand{\I}{\mathcal{I}}
\renewcommand{\L}{\mathcal{L}}
\newcommand{\ns}[1]{\left|\!\left|\!\left| {#1}\right|\!\right|\!\right|}
\newcommand{\bfphi}{{\boldsymbol \phi}}

\def\bfe{{\bf e}}
\def\hati{{\hat \imath}}
\def\hatj{{\hat \jmath}}
\def\bft{{\bf t}}
\def\bfu{{\bf u}}
\def\bfk{{\bf k}}
\numberwithin{equation}{section}

\title{Polyharmonic and Related Kernels on Manifolds: Interpolation
  and Approximation  \thanks{\emph{Key
      words:} manifold, positive definite kernels, least squares
    approximation, Sobolev spaces}}

\author{T.~Hangelbroek\thanks{ Department of Mathematics, Vanderbilt
    University, Nashville, TN 37240, USA. Research supported
    by  by grant DMS-1047694 from the National
    Science Foundation.}, 
F. J. Narcowich\thanks{ Department of Mathematics, Texas A\&M
    University, College Station, TX 77843, USA. Research
    supported by grant DMS-0807033 from the National
    Science Foundation.}, 
J. D. Ward\thanks{ Department of Mathematics, Texas A\&M University,
    College Station, TX 77843, USA. Research supported by
    grant DMS-0807033 from the National Science
    Foundation.}}

\begin{document}
\maketitle

\begin{abstract}
  This article is devoted to developing a theory for effective kernel
  interpolation and approximation in a general setting. For a wide
  class of compact, connected $C^\infty$ Riemannian manifolds,
  including the important cases of spheres and $SO(3)$, we establish,
  using techniques involving differential geometry and Lie groups,
  that the kernels obtained as fundamental solutions of certain
  partial differential operators generate Lagrange functions that are
  uniformly bounded and decay away from their center at an algebraic
  rate, and in certain cases, an exponential rate.  An immediate
  corollary is that the corresponding Lebesgue constants for
  interpolation as well as for $L_2$ minimization are uniformly
  bounded with a constant whose only dependence on the set of data
  sites is reflected in the \emph{mesh ratio}, which measures the
  uniformity of the data. The kernels considered here include the
  restricted surface splines on spheres, as well as surface splines
  for $SO(3)$, both of which have elementary closed-form
  representations that are computationally implementable.
  In addition to obtaining bounded Lebesgue constants in this setting,
  we also establish a ``zeros lemma'' for domains on compact
  Riemannian manifolds -- one that holds in as much generality as the
  corresponding Euclidean zeros lemma (on Lipschitz domains satisfying interior
  cone conditions) with constants that clearly demonstrate the influence
  of the geometry of the boundary (via cone parameters) as well as that of the 
  Riemannian metric.
\end{abstract}


\section{Introduction}\label{sec0}

Radial basis functions (RBFs) have proven to be a powerful tool for
analyzing scattered data on $R^n$. More recently, spherical basis
functions (SBFs), which are analogs of RBFs on the $n$-sphere, and
periodic basis functions (PBFs), which are analogs of RBFs on the
$n$-torus, have had comparable success for analyzing scattered data on
these manifolds.  A theoretical drawback is that most RBFs are
globally defined, thin plates splines are an example, and even those
that are locally defined such as Wendland functions behave globally
when approximating at densely packed data sites. Nevertheless, certain
RBF approximants, in their numerical implementation, exhibit localized
behavior, i.e., changing data locally only significantly alters the
interpolant locally. Since the pioneering work of Duchon \cite{D1,D2}, it
has been a mystery why RBF/SBF approximants display local behavior
even though the bases are globally supported. It was long suspected
that there were ``local bases'' hidden within the space of translates
of RBFs/SBFs.

A major objective of this paper is to establish that, for spheres and
$SO(3)$, there are closed-form kernels whose associated approximation
spaces possess highly localized bases, in the form of Lagrange
functions for given scattered data. Our previous work \cite{HNW}
established the existence of such bases on compact manifolds, but the
kernels we constructed did not have closed form. 

We will carry out the construction and proofs that establish the
existence and properties of these closed-form kernels in the context
of of more general manifolds, thus achieving another objective:
obtaining results for a broader class of kernels on manifolds than the
ones we treated in \cite{HNW}, where only reproducing kernels for
Sobolev spaces were dealt with.

We also address similar issues for conditionally positive kernels on
manifolds, where a given space of functions is to be reproduced. In
the case of $\RR^d$, this involves little more than polynomial
reproduction. For manifolds -- even for spheres and $SO(3)$ -- the
kernels and spaces are not so simple, and new techniques are required
to deal with the problem.


\paragraph{Goals} Given a manifold $\M$, a finite set of points $X=
\{x_1,\dots,x_N\} \subset \M$ and a kernel $k\colon\ \M \times \M \to
\reals$, one may attempt to fit a smooth function using functions from
$V_X := \spam\{k(\cdot,x_j),\ x_j\in X\}$, or more generally, to use
functions of the form
\begin{equation}\label{basic_form}
s = \sum_{j=1}^N a_{j} k(\cdot,x_j)  +p,
\end{equation}
where the supplementary function $p$ comes from a simple space (like
polynomials or spherical harmonics). The framework described above
applies to fitting data by means of interpolation, least squares or
near interpolation with a smoothing term.

This article is devoted to developing a theory for effective kernel
approximation in a general setting.  The problem is described as
follows: we seek kernels $k\colon\ \M\times \M \to \reals$ for which
interpolation is well posed and that have a convenient closed form
representation allowing for effective computations. Furthermore, we
are interested in aspects of the interpolants/approximants concerning
stability, locality and so on. 

In \cite{HNW} and \cite{HNSW}, we developed a theory for compact
Riemannian manifolds using positive definite ``Sobolev kernels.'' The
theory developed there addresses and answers questions concerning
properties of bases for $V_X$, properties related to locality,
stability of approximation and interpolation, and other matters. In
this paper, these questions, which are listed below, are addressed and
answered for a broad class of kernels on $\M$ that are Green's
functions for certain elliptic operators, and, when the manifold is a
sphere, real projective space or $SO(3)$, are computationally
implementable as well.

{\bf Locality}. Are there local bases for $V_X$? That is, are there
bases similar to those for wavelet systems or B-splines \cite[Chapter
5]{DL}? At a minimum, we would like a basis $\{v_j\}$ to satisfy
$|v_j(x)| \le r(d(x,x_j))$, with $r$ a rapidly decaying function.

{\bf $L_p$ conditioning}. Are there bases that are well conditioned in
$L_p$, after renormalization? That is, can we find bases for which
there constants $c_1,c_2$ such that $c_1 \|a\|_{\ell_p}\le
\|\sum_{j=1}^N a_j v_j\|_{L_p} \le c_2 \|a\|_{\ell_p}$, with $c_1,c_2$
independent of $N$, and, after a suitable normalization, independent of
$p$?

{\bf Marcinkiewicz-Zygmund property}. Does the 
space $V_X$
possess a Marcinkiewicz-Zygmund property relating samples to the size
of the function?  For $s\in V_X$, this means that the norms
$\|j\mapsto s(x_j)\|_{\ell_p}$ and $\|s\|_{L_p}$ are equivalent, with
constants involved independent of $N$.

{\bf Stability of interpolation}. Is interpolation  stable? Is the
Lebesgue constant bounded or, more generally, is the $p$ norm of the
interpolant controlled by the $\ell_p$ norm of the data?

{\bf Stability of approximation in $L_p$}. Is approximation by $L_2$
projection stable in $L_p$? Here, $1\le p\le \infty$.  In particular,
what we want is that the orthogonal projector with range $V_X$ be
continuously extended to each $L_p$, and that it has bounded operator
norm independent of $N$.

The Sobolev kernels we dealt with in \cite{HNW} do not possess simple,
closed form representations, even when the underlying manifold is a
sphere; they are defined indirectly, as reproducing kernels for
certain Sobolev spaces. To applied effectively to data fitting
problems, such a kernel should have an implementable characterization,
by which interpolation, approximation or other computational problems
can be treated. In the important cases relating to spheres and
$SO(3)$, we exhibit computationally implementable kernels. In
particular, these kernels include restricted surface splines on
spheres, and surface splines on $SO(3)$, both of which have simple
closed form representations.  Furthermore, for both of these cases,
theoretical approximation results concerning direct theorems, inverse
theorems, and Bernstein inequalities are known to hold \cite{MNPW}. In
conjunction with the stability of interpolation and least-squares
approximation in $L_p$, both yield new, precise error estimates for
these implementable schemes.

\paragraph{Kernels} The class of kernels considered in this paper are
those kernels $\kappa$ that act as fundamental solutions for elliptic
differential operators of the form $\L_m = \prod_{j} (\Delta - r_j)$
and lower order perturbations of these.  The origin of this approach
lies in the work of Duchon, \cite{D1,D2} on surface splines, where the
underlying kernel is the Green's functions for iterated Laplacian,
$\Delta^m$, on $\reals^d$.  Such kernels have also been used on
Riemannian manifolds, \cite{DNW, Pesenson00}.  For this reason, we
call them \emph{polyharmonic}; see
Definition~\ref{polyharmonic_kernels}.  Throughout this article, we
use $k_m$ to denote a generic polyharmonic kernel. 

This is a classic family of kernels and is sufficiently robust to
include many interesting examples. For instance, such kernels have
also been in use for some time on spheres, and have formed one of the
earliest families of SBFs (see \cite{Freeden} and references for a
list of early examples). In this setting, certain careful choices of
these kernels result in the complete family of surface splines
restricted to the sphere\footnote{A related problem, which can be
  considered a generalization of this particular set up has recently
  been considered by Fuselier and Wright, \cite{FusWrt}. There, kernel
  interpolation is considered on manifolds that are embedded in
  $\reals^d$ by using the restriction of various other RBFs to the
  manifold -- this is accomplished by constructing interpolants in the
  ambient Euclidean space and then restricting these to the
  manifold. (In contrast, we work directly with the manifold, making
  use of its intrinsic structure.)}, introduced in \cite{NSW}, which
we define below in (\ref{def_rss}) and denote in the ``zonal'' form as
$k_m(x,t) = \phi_s(x\cdot t) $.  Here $s$ is related to $m$ via $m = s
+d/2$.

It also includes the surface splines on $SO(3)$, introduced in
\cite{HS} and defined below in (\ref{def_so3}), and denoted throughout
the paper by $\bfk_m$.

A second type of kernel, ostensibly different from the polyharmonic
kernels, are the Sobolev (or Mat{\'e}rn) kernels, which have been
introduced for compact Riemannian manifolds in \cite{HNW}.  These
kernels come about as reproducing kernels for Sobolev spaces.  We
denote such kernels by $\kappa_m(x,y)$, where $m$ indicates the order
of the Sobolev space.  A corollary of the results presented in this
paper is that, in many cases, the Sobolev kernels are in fact
polyharmonic kernels. There is an operator $\L_m$ for which
$\kappa_m(x,y)$ is the fundamental solution.

\begin{table}[ht]
\begin{center}
\begin{tabular}{|l|c|c|}
\hline
{\bf Kernel	
}		
& {\bf Notation	}
& {\bf Location in manuscript	}
\\
 \hline
Restricted surface spline	
& $(x,y)\mapsto \phi_s(x\cdot y)$	
&Example 2 in \ref{ss_examples} 	\\ 
\hline
Surface spline on SO(3)	& 
$\bfk_m$		&
Example 3 in \ref{ss_examples} \\ \hline
Polyharmonic kernel			& 
$k_m$			&
Definition \ref{polyharmonic_kernels}	\\ \hline
Sobolev kernel	&
$\kappa_m$&
\cite[3.3]{HNW}	\\ \hline
\end{tabular}
\caption{Index for kernels used}
\end{center}
\end{table}
\paragraph{Outline} The layout of this paper is as follows. Following
the introduction and background, Section~\ref{background} deals with
certain geometric notions relevant to this article.
Section~\ref{interpolation_kernels} treats interpolation by {\em
  conditionally positive definite} kernels, the function spaces that
they generate, and the nature of their interpolants. We discuss some
important examples on well known manifolds, including spheres and the
rotation group. Finally we define precisely the polyharmonic kernels,
which are the kernels that we treat in our main results; they include
the examples we provided earlier. We demonstrate that they are
conditionally positive definite, identify the seminorm of the native
spaces associated with these kernels, and discuss the variational
problem associated with their interpolants.

The relationship between the polyharmonic kernels and the Sobolev
kernels of \cite{HNW} will be covered in
Section~\ref{quadratic_forms}. We show that, under certain conditions,
the polyharmonic operators $\L_m$ can be expressed as combinations of
operators generated by covariant derivatives, and vice versa (this is
done in Lemma~\ref{Lap_Cov}).  This allows us to conclude that Sobolev
kernels are examples of polyharmonic kernels.  On the other hand, it
permits us to demonstrate, in Section~\ref{ss_LBzeros}, that the
native space seminorms associated with polyharmonic kernels exhibit
the same behavior (metric equivalence to Euclidean Sobolev seminorms,
zeros lemmas, etc.)  as the native space norms associated with Sobolev
kernels.

The main results of the paper are given in
Section~\ref{lagrange_function}. Namely, the Lagrange function
associated with a kernel $k_m$ is rapidly decaying, and gives rise to
a uniformly bounded Lebesgue constant and a uniformly bounded $L_2$
minimization projector. The properties mentioned above -- concerning
locality, stability, conditioning, and so on -- then follow
immediately. We then discuss implications of this for surface spline
kernels on spheres and $SO(3)$.


Essential to our proofs in Section~\ref{lagrange_function} are theorems
giving $L_p$ Sobolev-space estimates for functions having zeros
quasi-uniformly distributed on a domain $\Omega$, with $\partial
\Omega$ being Lipschitz. Such theorems may hold both in $\reals^d$ and 
on $\M$ itself, and in
Section A, we treat both cases. For the case of a manifold $\M$,
these theorems involve geometric ideas; in particular, they require use of a
\emph{minimal $\varepsilon$-set} of points in $\M$
(cf. \cite{grove-petersen-1988-1}), which replaces a simpler set in
$\reals^d$. The results for $\M$ turn out to be intrinsic and hold in
the same generality as those in the Euclidean case. The bounds and the
condition on the meshnorm reflect geometric properties --
particularly, parameters from the cone condition on $\partial \Omega$, 
properties of the manifold $\M$, 
and parameters of the Sobolev spaces themselves -- but are independent of the volume and
diameter of $\Omega$.

\section{Background}
\label{background}
We now discuss some relevant details about analysis on compact,
complete, connected $C^\infty$ Riemannian manifolds. This is the same
setting as \cite{HNW}. Refer to it for a more detailed treatment and
further references.

Throughout our discussion, we will assume that $(\M,g)$ is a
$d$-dimensional, connected, complete $C^\infty$ Riemannian manifold
without boundary; the Riemannian metric for $\M$ is $g$, which defines
an inner product $g_p(\cdot,\cdot )=\langle\cdot,\cdot \rangle_{g,p}$
on each tangent space $T_pM$; the corresponding norm is
$|\cdot|_{g,p}$. As usual, a chart is a pair $(\fraku, \phi)$ such
that $\fraku\subset \M$ is open and the map $\phi:\fraku \to \reals^d$
is a one-to-one homeomorphism. An atlas is a collection of charts
$\{(\fraku_\alpha,\phi_\alpha)\}$ indexed by $\alpha$ such that
$\M=\cup_\alpha \fraku_\alpha$ and, when
$\phi_\alpha(\fraku_\alpha)\cap\phi_\beta(\fraku_\beta)\ne \emptyset$,
$\phi_\beta \circ\phi_\alpha^{-1}$ is $C^\infty$. In a fixed chart
$(\fraku, \phi)$, the points $p\in \fraku$ are parametrized by
$p=\phi^{-1}(x)$, where $x=(x^1,\ldots,x^d)\in U=\phi(\fraku)$.

In these local coordinates, $TM_p$ and $T^\ast M_p$, the tangent and
cotangent spaces at $p$, have bases comprising the vectors $\bfe_j =
\left(\frac{ \partial}{\partial x^j}\right)_p$, $j=1\ldots d$ and
$\bfe^k= (dx^k)_p$, $k=1\ldots d$, respectively. These vary smoothly
over $U=\phi(\fraku)$ and form dual bases in the sense that
$\bfe^k(\bfe_j) = \frac{ \partial x^k}{\partial x^j} = \delta^k_j$. In
the usual way, the inner product $\langle\cdot,\cdot \rangle_{g,p}$
provides an isomorphism between the cotangent and tangent
spaces. Thus, regarding $\bfe^k$'s as vectors, we have that $\langle
\bfe^k,\bfe_j \rangle_{g,p} = \delta^k_j$. A vector $\bfv$ can be
represented either as $\bfv=\sum_j v^j\bfe_j$ or as $\bfv = \sum_k
v_k\bfe^k$; the $v^j$'s and $v_k$'s are the \emph{contravariant} and
\emph{covariant} components of $\bfv$, respectively. Relative to these
bases, the inner product $\langle\cdot,\cdot \rangle_{g,p}$ has the
form
\begin{equation}
\label{metric_co_contra_form}
\langle\bfu,\bfv \rangle_{g,p} = \sum_{i,j=1}^d g_{ij}u^i v^j =
\sum_{i,j=1}^d g^{ij}u_i v_j\,, \ \text{where }g_{ij}= \langle \bfe_i,
\bfe_j \rangle_{g,p}\ \text{and } g^{ij}= \langle \bfe^i, \bfe^j
\rangle_{g,p}
\end{equation}
The matrices $(g_{ij})$ and $(g^{ij})$ are inverse to each other, and
are of course symmetric and positive definite. The inner product
$\langle\bfv,\bfw \rangle_{g,p}$ is itself independent of
coordinates. In addition, if $\bfv$ and $\bfw$ are $C^\infty$ vectors
fields in $p$, then $\langle\bfv,\bfw \rangle_{g,p}$ is also
$C^\infty$. 

The metric $g$ also induces an invariant volume measure $d\mu$ on
$\M$. The local form of the measure is $d\mu(x) =
\sqrt{\det(g)}dx^1\cdots dx^d$, where $\det(g) =
\det(g_{ij})$. 

\emph{Geodesics} are curves $\gamma : \RR \to \M$ that locally
minimize the arc length functional, $\int_a^b |\dot
\gamma|_{g,p}dt$. If we use the arc length $s$ as the parameter (i.e.,
$t\to s$), then, in local coordinates, a geodesic satisfies the
Euler-Lagrange equations:
\begin{equation}
  \label{geodesic_christoffel}
\frac{d^2 x^k}{ds^2} +  
\sum_{i,j=1}^d\Gamma^k_{ij} \frac{dx^i}{ds}\frac{dx^j}{ds} = 0, \quad
\text{where }\    
\Gamma^k_{ij} := \frac12
\sum_{m\in \D}g^{km} \left( \frac{\partial g_{jm}}{\partial x^i} + 
  \frac{\partial g_{im}}{\partial x^j} - \frac{\partial
    g_{ij}}{\partial x^m}
\right).
\end{equation}
The $\Gamma^k_{ij}$ are the \emph{Christoffel symbols}.

A geodesic solving (\ref{geodesic_christoffel}) is specified by giving
an initial point $p\in \M$, whose coordinates we may take to be
$\bigl(x^1 (0),\dots,x^d (0)\bigr)=0$, together with a tangent vector
$\bft_p$ having components $\frac{dx^i}{ds}(0)$. A Riemannian manifold
is said to be \emph{complete} if the geodesics are defined for all
values of the parameter $s$. All compact Riemannian manifolds without boundary are
complete, and so are many non-compact ones, including $\reals^d$.

We define the \emph{exponential map} $\Exp_p: T_pM \to \M$ by letting
$\Exp_p(0)=p$ and $\Exp_p(s\bft_p)=\gamma_p(s)$, where $\gamma_p(s)$
is the unique geodesic that passes through $p$ for $s=0$ and has a
tangent vector $\dot \gamma_p(0)=\bft_p$ of length 1; i.e., $\langle
\bft_p,\bft_p\rangle_{gp}.=1$.

Although geodesics having different initial, non-parallel unit tangent
vectors $\bft_p=\dot \gamma_p(0)$ may eventually intersect, there will
always be a neighborhood $\fraku_p$ of $p$ where they do not.  In
$\fraku_p$, the initial direction $\bft_p$, together with the arc
length $s$, uniquely specify a point $q $ via $q=\gamma_p(s)$, and the
exponential map $\Exp_p$ is a diffeomorphism between the corresponding
neighborhoods of $0$ in $T_pM$ and $p$ in $M$. In particular, there
will be a largest ball $B(0,\mathrm{r}_p)\in T_pM$ about the origin in
$T_pM$ such that $\Exp_p: B(0,\mathrm{r}_p)\to
\b(p,\mathrm{r}_p)\subset \M$ is injective and thus a diffeomorphism;
$\mathrm{r}_p$ is called the \emph{injectivity radius for $p$}. By
choosing cartesian coordinates on $B(0,\mathrm{r}_p)$, with origin
$0$, and using the exponential map, we can parametrize $\M$ in a
neighborhood of $p$ via $q=\Exp_p(x)$, $x\in T_pM$.

The \emph{injectivity radius} of $\M$ is $\inj:=\inf_{p\in
  \M}\mathrm{r}_p$. If $0<\inj\le \infty$, the manifold is said to
have \emph{positive} injectivity radius. For any $r< \inj$ and any
$p\in \M$, the exponential map $\Exp_p: B(0,r)\to \b(p,r)$ is a
diffeomorphism.

We make special note of the fact that, for a compact Riemannian
manifold, the family of exponential maps are uniformly isomorphic;
i.e., there are constants $0<\Gamma_1\le 1 \le \Gamma_2<\infty$ so that for
every $p_0\in \M$ and every $x,y\in B(0,\mathrm r)$, where $\mathrm r
\le \inj/3$,
\begin{equation}\label{isometry}
 \Gamma_1|x-y|\le \d(\Exp_{p_0}(x), \Exp_{p_0}(y)) \le \Gamma_2 |x-y|.
\end{equation}
When we use $\Gamma_1$ and $\Gamma_2$ in this paper, we always assume
that there is some fixed radius smaller that $\inj$ on which they are
computed. This avoids a tiresome repetition of this fact throughout
the paper.

An order $k$ \emph{covariant tensor} $\mathbf T$ is a real-valued,
multilinear function of the $k$-fold tensor product of $T_pM$. We
denote by $T_p^kM$ the covariant tensors of of order $k$ at $p$. In
terms of the local coordinates, there is a smoothly varying basis
$\bfe^{i_1}\otimes \dots\otimes \bfe^{i_k}$ for the $k$-fold tensor
product of tangent spaces. Thus, the covariant tensor field $\mathbf
T$ of order $k$ on $\fraku$ can be written as
\[
\mathbf T = \sum_{\hati \in \D^k} T_{\hati}\, \bfe^{i_1}\otimes
\cdots\otimes \bfe^{i_k},
\]
where we adopt the convention $\hati = (i_1,\dots,i_k)$. The
$T_{\hati}$ are the covariant components of $\mathbf T$. The metric
$g_{ij}$ is itself an order 2 covariant tensor field. One can also
define contravariant tensors and tensors of mixed type.

Because $T_p^kM = T_pM\otimes \cdots\otimes T_pM$ (k times), the
metric $g$ induces a natural, useful, invariant inner product on
$T_p^kM$; in terms of covariant components, it is given by
\begin{equation}
\label{metric_inner_prod_tensor}
\langle \mathbf S, \mathbf T\rangle_{g,p} = 
\sum_{\hati,\hatj \in \D^k} g^{i_1 j_1}\cdots g^{i_k j_k}S_\hati \,T_\hatj\,.
\end{equation}

The \emph{covariant derivative}, or \emph{connection}, $\nabla$
associated with $(\M,g)$ is defined as follows. If $\mathbf
T$ is an order $k$ (covariant) tensor, then the covariant derivative
of $\mathbf T$ is
\[
\nabla \mathbf T = \sum_{j\in\D} \sum_{\hati \in \D^k}
\bigg(\underbrace{ \frac{\partial T_{\hati}}{\partial x^j} - \sum
  _{r=1}^k \sum_{s\in \D} \Gamma^s_{j, i_r} T_{i_1,\ldots,i_{r-1},s,
    i_{r+1},\ldots, i_k}}_{(\nabla \mathbf T)_{\hati,j}} \bigg)
\bfe^{i_1}\otimes \cdots \otimes \bfe^{i_k}\otimes \bfe^j,
\]
which is an order $k+1$ covariant tensor with components $(\nabla
\mathbf T)_{\hati,j}$. The $\Gamma^k_{ij}$ are the Christoffel symbols
defined earlier.

A smooth function $f:\M \to \RR$ is a $0$ order tensor and so $\nabla
f$, which is the gradient of $f$ is an order $1$ tensor, $\nabla^2 f$
is an order 2 tensor. Continuing in this way, we may form $\nabla^k
f$, which is an invariant version of the ordinary $k^{th}$ gradient of
a function on $\RR^d$. (The superscript $k$ here is an operator power,
not a contravariant index.) The components of the $k^{th}$ covariant
derivative of $f$ have the form
\begin{equation}
\label{kth_order_covariant_deriv_tensor}
(\nabla^k f(x))_{\hati} 
= 
(\partial^k f(x))_{\hati} 
+ 
\sum_{m=1}^{k-1}
  \sum_{\hatj \in \D^m}
    A_{\hati}^{\hatj}(x) (\partial^m f(x))_{\hatj} 
\end{equation}
where 
\[
(\partial^m f)_{\hatj} := \frac{\partial^m }{\partial
  x^{j_1}\cdots\partial x^{j_m} } f\circ \phi^{-1} ,
\]
and where the coefficients $x\mapsto A_{\hati}^{\hatj}(x)$ depend on
the Christoffel symbols and their derivatives to order $k-1$, and,
hence, are smooth in $x$.  This can also be written in standard
multi-index notation. Let $\alpha_1, \alpha_2,\ldots, \alpha_d$ be the
number of repetitions of $1,2,\ldots, d$ in $\hatj $, and let $\alpha
= (\alpha_1,\ldots,\alpha_d)$. Then,
\begin{equation}
\label{multi_index_superscript_notation}
(\partial^m f)_{\hatj} :=\frac{\partial^m }{\partial (
  x^{1})^{\alpha_1}\cdots\partial  (x^{d})^{\alpha_d} } f\circ
\phi^{-1} =:  
D^{|\alpha |}_\alpha f\circ \phi^{-1} , \  
|\alpha | =\sum_{k=1}^d \alpha_k =m.
\end{equation}

Another important quantity that we need to deal with is the
\emph{adjoint} of the covariant derivative $\nabla^\ast$. This
operator is defined by $\int_\M \langle \nabla^\ast\mathbf T, \mathbf
S\rangle_{g,p}d\mu=\int_\M \langle \mathbf T, \nabla \mathbf
S\rangle_{g,p}d\mu$, where the inner product is given by
(\ref{metric_inner_prod_tensor}), and it takes an order $k+1$ tensor
to an order $k$ tensor. The coordinate form of $\nabla^\ast \mathbf
T$ is obtained via integration by parts:
\begin{equation}
\label{adj_cov_deriv}
(\nabla^\ast \mathbf T)_{\hati} = -\sum_{i,j} g^{jk}(\nabla \mathbf
T)_{\hati,j,k}
\end{equation}

We can combine covariant derivatives and their adjoints to get scalar
operators. In particular, if $f:\M\to \RR$ is $C^\infty$, then
$\nabla^kf$ is an order $k$ tensor. By applying $\nabla^\ast$ to it,
we get $(\nabla^\ast)^k\nabla^k f$, which is a scalar. (Note that $
(\nabla^\ast)^k=(\nabla^k)^\ast$.) The \emph{Laplace--Beltrami
  operator} $\Delta := -\nabla^\ast\nabla$. In coordinates, again
letting $\det(g) = \det(g_{ij})$, we have that
\[
\Delta f = \det(g)^{-1/2}\sum_{i,j}\frac{\partial}{\partial
  x^i}\left(\det(g)^{1/2}g^{ij}\frac{\partial f}{\partial x^j}\right).
\]

\paragraph{Sobolev spaces on subsets of $\M$}
Sobolev spaces on subsets of a Riemannian manifold can be defined in
an invariant way, using covariant derivatives \cite{Aub}. In defining
them, we will need to make use of the spaces $L_\sfp$, $L_\sfq$. To
avoid problems with notation, we will use the sans-serif letters
$\sfp$, $\sfq$, rather than $p$, $q$, as subscripts. Here is the
definition.
\begin{definition}[{\cite[p.~32]{Aub}}]\label{Sob_Norm}
  Let $\Omega\subset \man$ be a measurable subset. For all $1\le \sfp
  \le \infty$, we define the Sobolev space $W_\sfp^m(\Omega)$ to be
  all $f:\M \to \reals$ such that, for $0\le k\le m$, $ | \nabla^k f
  |_{g,p} $ in $L_\sfp(\Omega)$. The associated norms are as follows:
%
%
\begin{equation}\label{def_spn} 
\|f\|_{W_\sfp^m(\Omega)}:= 
\left\{
\begin{array}{cl}
  \left( 
    \sum_{k=0}^m
    \int_{\Omega} 
    | \nabla^k f |_{g,p}^\sfp
    \, \dif \mu(p)\right)^{1/\sfp}, & 1\le \sfp <\infty; \\ [5pt]
  \max_{\,0\le k\le m} \bigl\|  | \nabla^k f |_{g,p}
  \bigr\|_{L_\infty(\Omega)}, &
  \sfp=\infty.
\end{array}
  \right.
\end{equation}
When $\sfp=2$, the norm comes from the Sobolev inner product
\begin{equation}\label{def_sn} 
\langle f, g\rangle_{m,\Omega}:=\langle f,g\rangle_{W_2^m(\Omega)}:= 
  \sum_{k=0}^m
  \int_{\Omega} 
    \left\langle
      \nabla^k f,  \nabla^k g
    \right\rangle_{g,p}
  \, \dif \mu(p).
\end{equation}
We also write the $\sfp=2$ Sobolev norm as $\|f\|_{m,\Omega}^2 :=
\langle f,f\rangle_{m,\Omega}$.  When $\Omega=\man$, we may suppress
the domain: $\langle f,g\rangle_m = \langle f,g\rangle_{m,\M}$ and
$\|f\|_m=\| f\|_{m,\M}$.
\end{definition}
\paragraph{Metric equivalence} The exponential map allows us to
compare the Sobolev norms we've just introduced, to standard Euclidean
Sobolev norms as follows:
\begin{lemma}[{\cite[Lemma~3.2]{HNW}}]\label{Fran}
  For $m\in\nats$ and $0<r<\inj/3$, there are constants $0 < c_1 <c_2$
  so that for any measurable $\Omega \subset B_{r}$, for all $j\in
  \nats$, $j\le m$, and for any $p_0\in \M$, the equivalence
$$
c_1\| u\circ\Exp_{p_0}\|_{W_\sfp^j(\Omega)}\le
\|u\|_{W_\sfp^j(\Exp_{p_0}(\Omega))}\le c_2\|
u\circ\Exp_{p_0}\|_{W_\sfp^j(\Omega)}
$$
holds for all $u:\mathrm{\Exp}_{p_0}(\Omega)\to \reals$. The constants
$c_1$ and $c_2$ depend on $r$, $m$ and $\sfp$, but they are
\emph{independent} of $\Omega$ and $p_0$.
\end{lemma}
\paragraph{Besov spaces on $\M$}
Besov spaces can be defined and characterized in many equivalent
ways. For a discussion, see Triebel's book \cite[1.11, and Chapter
7]{Trieb} and the references therein.  Our definition follows
Triebel's.
\begin{definition}\label{Besov}
  For $0<s\le m$ and $1\le p< \infty$, we define the Besov space
  $B_{p,\infty}^s\bigl(\M\bigr)$ as the collection of functions in
  $L_p\bigl(\M\bigr)$ for which the following expression
$$\|f\|_{B_{p,\infty}^s(\M)}:=\sup_{t>0} t^{-s} K(f,t) $$
is finite, where the $K$-functional $K(f,\cdot):(0,\infty)\to
(0,\infty)$ is defined as
$$K(f,t) := \inf\left\{\|f - g\|_{L_p} + t^{2m}\|g\|_{W_p^{m}(\M)}: g\in W_p^{m}(\M) \right\}.$$
For $p=\infty$, the definition is the same after substituting
 $L_{p}\bigl(\M\bigr)$ by $C\bigl(\M\bigr)$ and $W_p^{m}(\M)$ by $C^{m}(\M)$.
\end{definition}
\subsection{Notation}
In order to distinguish balls on $\reals^d$ from those in $\M$, 
we denote the ball centered at $p\in \M$ having radius $r$ by $\b(p,r).$ 
(Euclidean balls are denoted $B(x,r)$.)
Given a finite set $\Xi\subset\M$, we define its  \emph{mesh norm}  (or \emph{fill distance}) $h$ 
and the \emph{separation distance} $q$ to be:
\begin{equation} \label{minimal-separation}
 h:=\sup_{p\in \M} \d(p,\Xi)\qquad \text{and}\qquad  q:=
\inf_{\substack{\xi,\zeta\in \Xi\\ \xi\ne \zeta} }
\d(\xi,\zeta).
\end{equation}
The mesh norm measures the density of $\Xi$ in $\M$, the separation radius determines 
the spacing of $\Xi$. The \emph{mesh ratio} $\rho:=h/q$ measures the uniformity of the 
distribution of $\Xi$ in $\M$. If $\rho$ is bounded, then we say that the point set $\Xi$ is 
quasi-uniformly distributed, or simply that $\Xi$ is quasi-uniform.

\section{Interpolation by Kernels}
\label{interpolation_kernels}
The purpose of this section is to discuss further this interpolation
problem and to present the kernels we employ. The kernels we consider
are fundamental solutions of certain elliptic PDEs. They happen also
to be conditionally positive definite, a well known class for which
interpolation is understood. In particular, interpolation is well
posed, and has a dual nature, as best interpolation from a function
space.

In \ref{ss_cpd} we discuss interpolation with conditionally positive
definite kernels, and present the associated problem of best
interpolation. In \ref{ss_examples} we present some motivating
examples for our problem: surface spline interpolation on spheres and
on $SO(3)$. In \ref{ss_polyharmonic} we give the formal definition of
the kernels we use and the operators they invert; we also discuss the
associated variational problem they solve.

\subsection{Interpolation with conditionally positive definite
  kernels}
\label{ss_cpd}
The kernels we consider in this article are conditionally positive
definite on the compact Riemannian manifold. As a reference on this
topic, we suggest \cite[Section 4]{DNW}.

\begin{definition}\label{cpd}
A kernel is conditionally positive definite with
  respect to a finite dimensional space $\Pi$ if, for any set of
  centers $\Xi$, the matrix $\bigl(k(\xi,\zeta)
  \bigr)_{\zeta,\xi\in\Xi}$ is positive definite on the subspace of
  all vectors $\alpha\in \comps^{\Xi}$ satisfying $\sum_{\xi\in\Xi}
  \alpha_{\xi} p(\xi) = 0 $ for $p\in \Pi$.
\end{definition}

This is a very general definition which we will make concrete in the
next subsections. Given a complete orthonormal basis
$(\phi_j)_{j\in\nats}$, of continuous functions, normalized in
$L_{\infty}$ (i.e., $\|\phi_j\|_{\infty}= 1$) any kernel
$$k(x,y):=\sum_{j\in \nats} \tilde{k}(j) \varphi_j(x)\overline{\varphi_j(y)}$$
with coefficients $\tilde{k}\in \ell_1(\nats)$ for which all but
finitely many coefficients $\tilde{k}(j)$ are positive is
conditionally positive definite with respect to $\Pi_{\J} =
\spam(\phi_j \mid j\in \J),$ where $\J = \{j\mid \tilde{k}(j) \le
0\},$ since, evidently,
\begin{eqnarray*}
  \sum_{\xi\in\Xi} \sum_{\zeta\in\Xi}  \alpha_{\xi}
  k(\xi,\zeta)\overline{
    \alpha_{\zeta}}
  &=&
  \sum_{\xi\in\Xi} \sum_{\zeta\in\Xi}  \alpha_{\xi} \overline{\alpha_{\zeta}}
  \left( \sum_{j\in\nats} \tilde{k}(j) \phi_j(\xi)
    \overline{\phi_j(\zeta)} 
  \right) \\
  &=& 
  \sum_{j\in\nats}\tilde{k}(j)\sum_{\xi,\zeta\in\Xi} \alpha_{\xi} 
  \phi_j(\xi)\overline{ \alpha_{\zeta}\phi_j(\zeta) } =
  \sum_{j\notin \J} \tilde{k}(j) \|\alpha\phi_j\|_{\ell_2 (\Xi)}^2 >0
\end{eqnarray*}
provided $\sum_{\xi} \alpha_{\xi} \phi_j(\xi) = 0$ for $j$ satisfying
$\tilde{k}(j)\le 0$.

In this case if the set of centers $\Xi\subset \M$ is unisolvent with
respect to $\Pi_{\J} = \spam(\varphi_j\mid j\in\J)$ (meaning that
$p\in \Pi_{\J} $ and $p(\xi) = 0$ for $\xi \in\Xi$ implies that $p=0$)
then the system of equations
 $$
 \left\{
\begin{array}{ll}
  \sum_{\xi\in\Xi} a_{\xi} k(\zeta,\xi) + \sum_{j\in\mathcal{J}} b_j 
  \varphi_j(\zeta)= y_{\zeta}&\quad \zeta\in\Xi\\
  \sum_{\xi\in\Xi} a_{\xi} \varphi_j(\xi) = 0&\quad j\in\J
\end{array}
\right.
$$
has a unique solution in $\comps^{\Xi} \times\comps^{\J}$ for each
data sequence $\bigl(y_{\zeta}\bigr)_{\zeta\in \Xi}\in \comps^{\Xi}.$
When data is sampled from a continuous function at the points $\Xi$
(i.e., $y_{\zeta} = f(\zeta)$) this solution generates a continuous
interpolant:
$$
I_{\Xi}f =I_{k,\J,\Xi} f = \sum_{\xi\in\Xi} a_{\xi} k(\cdot,\xi) +
\sum_{j\in\mathcal{J}} b_j \varphi_j
$$
with the property that it is the minimizer of the  semi-norm $\ns{\cdot}_{k,\J}$, called the ``native space'' norm, given by
\begin{equation}\label{NS_norm}
\ns{ \sum_{j\in\nats}\hat{u}(j) \varphi_j}_{k,\J}^2 
= \sum_{j\notin \J} \frac{|\hat{u}(j)|^2}{\tilde{k}(j)},
\end{equation} 
over all functions $u =  \sum_{j\in\nats}\hat{u}(j) $ for which $u(\xi) = y_{\xi}$, $\xi\in \Xi$.
%
If $k$ is conditionally positive definite with respect to the set
$\Pi_{\J}$, it will be conditionally positive definite with respect to
$\Pi_{\J'}$ for any $\J'\supset \J$. For this reason, the interpolant
and norm both are decorated by $k$ and $\J$.

This has the consequence that any two conditionally positive definite
kernels $k,k'$ which have eigenfunction expansions that coincide on
all but finitely many indices (say $\I$), produce the same
interpolants. That is: $I_{k,\I,\Xi} = I_{k',\I,\Xi}$.

\subsection{Examples of conditionally positive definite
  kernels}\label{ss_examples}
\begin{example}[Surface Splines]\label{example_ss}
  As a first example of a conditionally positive definite kernel, we
  take $\M= \reals^d$, and consider the kernels $k_m(x,\alpha) =
  \bfphi_s(x - \alpha)$ given by the functions
$$\bfphi_s(x)=
C_{m,d}\begin{cases}
|x|^{2s}\log|x|\quad& d\, \text{is even}\\
|x|^{2s}\quad&d\,\text{is odd}
\end{cases}$$ where $s+d/2=m$.  For a certain $C_{m,d}\ne 0$, this is
a fundamental solution for the operator $ \Delta^m$.

Because of the positivity of the generalized Fourier transform, one can
see that $\bfphi_s$ is conditionally positive definite on $\reals^d$
with respect to $\Pi_{m-1}$.  These have been studied by Duchon
\cite{D1,D2}, and they comprise some of the earliest and most popular
examples of conditionally positive definite kernels.

Although our focus in this paper is on kernels on compact manifolds,
the family of surface splines acts as a useful benchmark, since they
have a simple, direct representation, as well as being conditionally
positive definite, not to mention that for certain interpolation
problems, their Lagrange functions decay rapidly (this was demonstrated in a least
squares sense by Matveev in \cite[Lemma 5]{Mat} and pointwise in \cite{HNW}) 
and have a uniformly bounded Lebesgue constant (cf. \cite{HNW}).
 \end{example}
 
\vspace{6pt}
\begin{example}[Restricted Surface Splines on
  $\sphere$] \label{example_rss} When $\M= \sphere$, the eigenvalues
  of the Laplace--Beltrami operator are $\mu_{\ell} = \ell(\ell+d-1)$
  and each eigenvalue has $N(d,\ell) =
  \frac{(2d+\ell)\Gamma(\ell+d-1)}{\Gamma(\ell+1)\Gamma(d)}$ linearly
  independent eigenfunctions, the spherical harmonics $Y_{\ell,m}$.

  We now introduce a family of kernels known as the restricted surface
  splines.  These are kernels indexed by $m\in \nats$, $m>d/2$. By
  writing $m=s+d/2$, we give the zonal expression
\begin{equation}\label{def_rss}
\phi_{s} (t):= 
\begin{cases}
|1-t|^s \log|1-t|&\quad 
s\in \nats\\
|1-t|^s&\quad 
s \in \nats +1/2.
\end{cases}
\end{equation}
When $d$ is even, $s$ is integral and the first formula is used. 
When $d$ is odd, the second is used. For a given $d$ and an
integer $m>d/2$, we write $k_m(x,y) = \phi_s(x\cdot y)$ to 
denote the corresponding kernel on $\sphere$.

A spherical harmonic expansion of the restricted surface splines can
be found in \cite[Equations (2.12) \& (2.20)]{BaHu}.  It is known
that, for $m>d/2$, $k_m(x,y) =
\sum_{\ell}\sum_{j}\widetilde{k}_m(\ell,j)Y_{\ell,j}(x)Y_{\ell,j}(y)$, 
where
\begin{equation}\label{Gegenbauer_expansion}
  \widetilde{k}_m(\ell,j)  = C_m \prod_{\nu=1}^m
  \left[\left(\ell+\left(\frac{d-1}{2}\right)\right)^2 - 
    (\nu-1/2)^2\right]^{-1},
  \quad \text{for } \ell>s.
\end{equation}
When $d$ is odd, this equation holds for all $\ell$.

From this formula, it follows that $k_m$ is conditionally positive
definite with respect to the space $\Pi_{\lfloor s\rfloor} := \spam
(Y_{\ell,j}\mid \ell\le s, j \le N(d,\ell))$.

A second consequence is that, by a possible slight correction of the
spherical harmonic expansion (discussed below), $k_m$ is the
fundamental solution for a differential operator of order $2m$ that is
polynomial in $\Delta$:
$$
\L_m:= C_m \prod_{\nu=1}^m \left[ \Delta - (\nu -
  d/2)(\nu+d/2+1)\right].
$$

We note that when $d$ is odd, the operator $\L_m$ is invertible on
$W_2^{2m}(\sphere)$. Indeed, it is nonvanishing on each spherical
harmonic $Y_{\ell,m}$.

When $d$ is even, the Fourier coefficients of the kernel follow
(\ref{Gegenbauer_expansion}) for $\ell> s$ only, but $\L_m$
annihilates spherical harmonics of degree $s$ or less. Indeed, in this
case, we can re-index the operator to get:
\begin{eqnarray*}
  \L_m &=& C_m  \prod_{\nu=1}^{d/2-1} \left[ \Delta - 
    (\nu - d/2)(\nu+d/2-1)\right] 
  \prod_{\nu=d/2}^{m} \left[ \Delta - (\nu - d/2)(\nu+d/2-1)\right]\\
  &=& C_m  \prod_{\nu=1}^{d/2-1} \left[ \Delta - (\nu - d/2)(\nu+d/2-1)\right] 
  \prod_{J=0}^{m-d/2} \left[ \Delta - J(J+d-1)\right].
\end{eqnarray*}
So $\L_m$ annihilates all the  spherical harmonics of order up to $s = m-d/2$.

In other words, for sufficiently smooth functions, say $f\in
C^{2m}(\sphere)$ represented by the series $f= \sum_{\ell=0}^{\infty}\sum_{m=1}^{N(d,\ell)}
\widehat{f}(\ell,j)Y_{\ell,j}$,
$$
f(x) = \int_{\sphere} \L_m [f ] (\alpha) \phi_s(x\cdot\alpha) \dif \mu
(\alpha) + p_f
$$
where we add a spherical harmonic term
$p_f=\sum_{\ell=0}^{s}\sum_{j=1}^{N(d,\ell)}
\widehat{f}(\ell,j)Y_{\ell,j}\in \Pi_{s}$ when $d$ is even (when $d$
is odd, $p_f=0$).
\end{example}

\begin{example}[Surface Splines on $SO(3)$]\label{example_so3}
When $\M= SO(3)$, the group of  proper rotations of $\reals^3$, 
the eigenvalues of the Laplace--Beltrami operator are  
$\mu_{\ell} = \ell(\ell+1)$ and each eigenvalue is associated with
$N(\ell) = (1+2\ell)^2$ linearly 
independent eigenfunctions, called Wigner D-functions and denoted by
$(D_{j,\iota}^{\ell})_{(|j|, |\iota|\le \ell)}$.

For $m\ge 2$ and $s= m-3/2$,
the surface spline kernels 
are
\begin{equation}\label{def_so3}
\bfk_m(x,y) = \left(\sin\left( \frac{\omega(y^{-1}x)}{2}\right)\right)^{2m-3},
\end{equation}
where $ \omega(z)$ is the rotational angle of $z\in SO(3)$.  The
corresponding Wigner D-function expansion $\bfk_m(x,y) =
\sum_{\ell}\sum_{j,\iota}\widetilde{\bfk}_m(\ell,j,
\iota)D_{j,\iota}^{(\ell)}(x)D_{j,\iota}^{(\ell)}(y)$ is discussed in
\cite[Lemma 2]{HS}, where it is shown that for some $C_m\ne 0$,
$$
\widetilde{\bfk}_m(\ell,j,\iota) = C_m \prod_{\nu=-(m-1)}^{m-1}
\left[\ell+ \nu +\frac{1}{2} \right]^{-1}.
$$ 
Thus, $\bfk_m$ is conditionally positive definite with respect to
the space 
$\Pi_{m-2} = \spam \{ D_{j,\iota}^{\ell}\mid \ell\le m-2,\, |j|,
|\iota|\le \ell\}$.

It also follows (from \cite[Lemma 3]{HS}) that $\bfk_m$ is the fundamental solution for the
differential operator of order $2m$ having the form:
$$
\L_m:= C_m \prod_{\nu=0}^{m-1} \left[\Delta - (\nu^2-1/4)\right] 
$$
in the sense that for $f\in C^{2m}$, $f = \sum_{\ell=0}^{\infty}
\sum_{|j|,|\iota|\le \ell}\widehat{f}(\ell,j,\iota)
D_{j,\iota}^{\ell}$ the formula
$$
f(x) = \int_{SO(3)} \L_m [f] (\alpha) \bfk_m(x,\alpha) \dif
\mu(\alpha)  
$$
holds true.
\end{example}

\subsection{Polyharmonic and related kernels}\label{ss_polyharmonic}

The kernels we wish to treat are fundamental solutions of differential
operators that are polynomial in the Laplace--Beltrami operator, or are
directly related to them. Since, on a compact Riemannian manifold
$\Delta$ is a self adjoint operator with a countable sequence of
nonnegative eigenvalues $\lambda_j\le \lambda_{j+1}$ having $+\infty$
as the only accumulation point, we can express the kernel in terms of
the associated eigenfunctions $\Delta \phi_j = \lambda_j \phi_j$.  We
now make this clear with a formal definition.
 
 
  \begin{definition}\label{polyharmonic_kernels}
   Let $m\in \nats$ such that $m>d/2$. We say that the kernel
   $k_m:\M\times\M\to \reals$ is polyharmonic if the following hold:
 \begin{enumerate}
 \item There exists a polynomial $Q(x) =\sum_{\nu=0}^m c_{\nu}
   x^{\nu}$ in $\Pi_m(\RR)$, with the highest order coefficient $c_m>
   0$, so that $Q(x)>0$ for all $x$ sufficiently large. Let the
   corresponding differential operator of order $2m>d$ be given by
$$
\L_m  = \sum_{\nu=0}^m c_{\nu} \Delta^{\nu} = Q(\Delta),
$$
and let $\J\subset \nats$ be a finite set that includes
all $j$ for which the eigenvalue $Q(\lambda_j)$ of $\L_m$ satisfies
$Q(\lambda_j)\le 0$. (In addition to this finite set, $\J$ may also include a finite number $j$'s for which $Q(\lambda_j)>0$.)  

\item The kernel has the eigenfunction expansion 
$k_m(x,y)=
  \sum_{j\in\nats} \tilde{k}_m(j) \psi_j(x)\psi_j(y)$, with
  coefficients $\tilde{k}_m(j) = 1/Q(\lambda_j)$ for $j\notin \J$. (On
  $\J$, $\tilde{k}_m(j)$ can assume arbitrary values.)
\end{enumerate}
\end{definition}

It follows immediately from this definition that $k_m$ is
conditionally positive definite with respect to the finite dimensional
space $\Pi_{\J}=\spam_{j\in\J}\varphi_j$. Another consequence is that, for $f\in C^{\infty}$,
 \begin{equation}\label{reproduction}
 f(x) = \int_{\M} \L_m [f - p_f](\alpha) k_m(x,\alpha) \dif \mu(\alpha)+p_f
 \end{equation}
 where $p_f= \sum_{j\in\J} \mathrm{proj}_j f$ is the orthogonal projection onto $\Pi_{\J}$. 
 
As previously stated, the interpolation operator $I_{\Xi,k_m,\J}$ produces the minimizer of the
seminorm $\ns{ u}_{k_m,\J}$. 
Since $k_m(x,y) = \sum_{j\notin\J} \tilde{k}_m(j) \varphi_j(x)\varphi_j(y)$
and, for $j\notin \J$, 
$$
\tilde{k}_m(j) = Q(\lambda_j)^{-1} = \left(\sum_{\nu=0}^m c_{\nu} (\lambda_j)^{\nu}\right)^{-1},
$$
which is the inverse symbol of $\L_m$,  it follows from (\ref{NS_norm}) that
$$
\ns{ u}_{k_m,\J}^2
=\sum_{j\notin \J} \frac{|\hat{u}(j)|^2}{\tilde{k}(j)}
=
\langle \L_m u, u\rangle_{L_2(\M)}
-\sum_{j\in\J} Q(\lambda_j)| \hat{u}(j)|^2.
$$
This relation connects the norm $\ns{ u}_{k_m,\J}$ with the quadratic form
 $ \langle \L_m u, u\rangle_{L_2(\M)}$. The goal of the next section is to study this quadratic form.
 
 \section{Operators and quadratic forms}
\label{quadratic_forms}
Of the two quadratic forms considered, the one derived from the native
space seminorm: $\ns{u}_{k_m,\J}^2$, and the one derived from the
operator $[u]_{}^2:=\langle\L_m u, u\rangle_{L_2(\M)}$, the latter has
much to offer from the point of view of analysis, but the former is
tied to the variational problem satisfied by the kernel
interpolants. The object of this section is to attain a better
understanding of $[u]_{}^2$.
 
To this end, we seek an analogue of the bilinear form $\langle \L
u,v\rangle_{L_2(\M)}$ -- one that is defined on measurable subsets of
$\M$.  A reasonable goal would be to find a form that is comparable to
the corresponding Sobolev form $\sum_{j=0}^m \langle
u,v\rangle_{j,\Omega}$, where $$\langle u,v\rangle_{j,\Omega} :=
\int_{\Omega} \langle \nabla^j u, \nabla^j v\rangle_{g,x} \dif
\mu(x).$$ This is the bilinear form used to define the Sobolev space
inner product:
 (\ref{def_spn}) of Definition \ref{Sob_Norm}
 for $\Omega \subset \M$.
 
 The rest of this section is structured as follows. In \ref{ss_LB} we
 demonstrate that on a wide class of manifolds, the elliptic operator
 composed of covariant and contravariant derivatives, which is at the
 heart of \cite{HNW}, is a polynomial in $\Delta$ and, conversely, the
 Laplace--Beltrami operator has an expansion in terms of these elliptic
 operators. This permits us immediately to classify the Sobolev
 kernels on spheres (as well-known kernels of a type studied in
 \cite{Hsphere,MNPW}) and to give concrete approximation results for
 them.  In \ref{ss_LBzeros} we present analogues to the bilinear form
 generated by $\L_m$ on measurable subsets. Using this, we demonstrate
 that this bilinear form behaves like a norm for functions with many
 zeros.
 
 \subsection{The Laplace--Beltrami operator and the covariant
   derivative}\label{ss_LB}
 Simply considering $\langle \L u,v\rangle_{L_2(\Omega)}$ for
 measurable subsets $\Omega \subset \M$ is not suitable, since there
 will be many functions for which $\L u$ may vanish on $\Omega$.  This
 is true even on $\reals^d$ when $\L= \Delta^m$. In this case there
 are many polyharmonic functions (and even harmonic functions!) on a a
 given subdomain $\Omega$ that may have nonzero Sobolev norms, despite the fact
 that they are in the kernel of $\L$.

 Guided by the observation that on $\reals^d$, $\int_{\reals^d}
 v\Delta^m u = \int_{\reals^d}\langle \nabla^m v\nabla^m u\rangle$ holds for test
 functions $u,v$, we first attempt to compare $\Delta^m$, the
 principle part of $\L$, to $(\nabla^m)^* \nabla^m$.  It is important
 to stress that $(\nabla^m)^*$ means the adjoint to $\nabla^m$, in the
 $L_2(\M)$ inner product, as defined in (\ref{adj_cov_deriv}).  \footnote{and not with respect to
   $L_2(\Omega)$ -- e.g., even though $\Delta^m = (-1)^m(\nabla^m)^*
   \nabla^m$ holds on $\reals^d$, it is not the case that
   $(-1)^m\int_{\Omega} v\Delta^m u = \int_{\Omega}\nabla^m v\nabla^m
   u$ for subsets $\Omega \subset \reals^d$.}
   To this end, we make the following assumption
   \begin{assumption} \label{Manifold_Assumption}
   For all $k\in \nats$, there exists a real polynomial $p_{k-1}$ of degree $k-1$, such that
\[
(\nabla^k)^* \nabla^k =(-1)^k\Delta^k + p_{k-1}(\Delta).
\]
   \end{assumption}

 A class of Riemannian manifolds that satisfy this are the
 \emph{two-point homogeneous spaces} \cite{Helgason_1984}, both
 compact and non-compact. A manifold $\M$ is homogeneous if $\M =
 G/K$, where $G$ is a Lie group and $K$ is a Lie subgroup of
 $G$. \emph{Two-point} homogeneous means that for any two pairs of
 points $p,q$ and $p',q'$ such that the distances $d(p,q) = d(p',q')$
 there is an isometry $\Phi\in G$ such that $p'=\Phi(p)$ and
 $q'=\Phi(q)$. These manifolds\footnote{We note that they have also
appeared in other approximation theory literature. See, e.g., \cite{BKLT, KLT}}
 have been completely classified (see
 \cite[p.\ 167 \& p.\ 177]{Helgason_1984} for lists), and include
 $\sph^d$ and the real projective spaces $\mathsf P^d$. (The rotation
 group $SO(3) =\mathsf P^3$.)

\begin{lemma}\label{homogeneous} Let $\M$ be a two-point homogeneous space.
Then $\M$ satisfies Assumption \ref{Manifold_Assumption}.
\end{lemma}
\begin{proof} The proof proceeds in two steps. The first is showing
  that if $\Phi:\M\to \M$ is a diffeomorphism that is also an isometry
  (i.e., preserves distances), then the operator $D:=(\nabla^k)^*
  \nabla^k$ is invariant in the sense that for any smooth function
  $f:\M\to \RR$, $Df = (D(f\circ \Phi))\circ \Phi^{-1}$. We will
  follow a technique used in \cite[Proposition~2.4,
  p.~246]{Helgason_1984}.  Let $(\fraku,\phi)$ be a local chart, with
  coordinates $x^j=\phi^j(p)$, $j=1,\ldots,n$ for $p\in \fraku$. Since
  $\Phi$ is a diffeomorphism, $(\Phi(\fraku), \phi\circ \Phi^{-1})$ is
  also a local chart. Let $\psi=\phi\circ \Phi^{-1}$, and use the
  coordinates $y^j=\psi^j(q)$ for $q\in \Phi(\fraku)$. The choice of
  coordinates has the effect of assigning the \emph{same} point in
  $\RR^n$ to $p$ and $q$, provided $q=\Phi(p)$ -- i.e., $x^j(p) =
  y^j(q)$.  Thus, relative to these coordinates the map $\Phi$ is the
  identity, and consequently, the two tangent vectors
  $(\frac{\partial}{\partial y^j})_q\in T_q\M $ and
  $(\frac{\partial}{\partial x^j})_p\in T_p\M$ are related via
\[
\left(\frac{\partial}{\partial y^j}\right)_{\Phi(p)} =
d\Phi_p\left(\frac{\partial}{\partial x^j}\right)_p.
\]
So far, we have only used the fact that $\Phi$ is a
diffeomorphism. The map $\Phi$ being in addition an isometry then
implies that
\[
\left\langle \frac{\partial}{\partial y^j}, \frac{\partial}{\partial
    y^k} \right\rangle_{\Phi(p)} = \left\langle
  d\Phi_p\left(\frac{\partial}{\partial x^j}\right),
  d\Phi_p\left(\frac{\partial}{\partial x^k}\right)
\right\rangle_{\Phi(p)}= \left\langle \frac{\partial}{\partial x^j},
  \frac{\partial}{\partial x^k} \right\rangle_p.
\]
The expression on the left is the metric tensor at $\Phi(p)$,
$g^{jk}(y)$, and on the right, the metric tensor at $p$,
$g^{jk}(x)$. The equation above implies that, as \emph{functions} of
$y$ and $x$, $g^{jk}(y) = g^{jk}(x)$. This means that the expressions
for the Christoffel symbols, covariant derivatives and various
expressions formed from them will, as functions, be the same. Since
the operator $D = (\nabla^k)^* \nabla^k$ is constructed from such
objects, it follows that $Df = (D(f\circ \Phi))\circ \Phi^{-1}$, and
so $D$ is invariant.

The second step makes use of two-point homogeneity. Since
$(\nabla^k)^* \nabla^k$ is invariant under every isometry $\Phi$ in
$G$, applying \cite[Proposition 4.11, p.~288]{Helgason_1984} yields
the result that $(\nabla^k)^* \nabla^k$ is a polynomial in the
Laplace--Beltrami operator: $(\nabla^k)^* \nabla^k = a_{k+1}\Delta^k
+a_k\Delta^{k-1} +\cdots a_0$. Comparing terms in the highest order
derivatives involved in coordinate expressions for both sides shows
that $a_{k+1}=(-1)^k$.
\end{proof}

Induction ensures that
 \[
 \Delta^k - (-1)^k(\nabla^k)^* \nabla^k = c_{k-1} (\nabla^{k-1})^*
 \nabla^{k-1} + c_{k-2} (\nabla^{k-2})^* \nabla^{k-2} \dots +c_0.
 \]
From this we have the following.
\begin{lemma}\label{Lap_Cov} 
  Suppose $\M$ is a Riemannian manifold satisfying Assumption
  \ref{Manifold_Assumption}.  If $Q(x) = c_m x^m + \dots +c_0$ is a
  (real) polynomial of degree $m$, then there exist real numbers
  $a_j$, with $a_m = c_m,$ so that
$$
Q(\Delta) = \sum_{j=0}^m a_j (\nabla^j)^* \nabla^j.
$$
Conversely, for any constants $b_j$, there is a real polynomial $p$
for which the operators $p(\Delta)$ and $\sum_{j=0^n} b_j
(\nabla^j)^{*}\nabla^j$ coincide.
\end{lemma}

An immediate consequence is that the Sobolev kernels $\kappa_{m,\M}$
considered in \cite{HNW} and \cite{HNSW} are Green's functions for
operators of the form $Q(\Delta)$, with $Q$ a real polynomial of
degree $m$.

We note, furthermore, that because the lead coefficient $c_m$ of $Q$  is
assumed positive (see Definition \ref{cpd}), we have that $a_m>0$.

\subsection{Connecting the quadratic form to the Sobolev norm}
\label{ss_LBzeros}
The benefit of Lemma \ref{Lap_Cov} is that we can use it to obtain
useful local versions of the form $\langle \L
u,v\rangle_{L_2(\M)}$. In particular, we consider, for $\L_m =
Q(\Delta)$, coefficients $a_0,\dots, a_m$ as guaranteed by Lemma
\ref{Lap_Cov}.  When $\Omega\subset \M$, the form $ \langle \sum_{j =
  0}^m a_j ( \nabla^j)^*\nabla^j u,v\rangle_{L_2(\Omega)} = \sum_{j =
  0}^m a_j \langle u,v\rangle_{j,\Omega} $ is the local version of
$\langle \L u,v\rangle_{L_2(\M)} $.  Indeed, we have
$$
[u]_{m,\Omega}^2 := \int_{\Omega} \beta(u,u)_x \dif \mu
$$
where 
$\beta(\cdot,\cdot)_x:C^{\infty}\times C^{\infty} \to \reals$ 
is the bilinear form
$\beta(u,v)_x = \sum_{j = 0}^m a_j \langle \nabla^j u,\nabla^j v\rangle_{x} $.
Clearly  for 
\begin{equation}\label{bars}
\overline{a} := \max_{j\le m}|a_j|
\qquad \text{and}\qquad 
\underline{a} = \max_{j\le m-1}|a_j|
\end{equation}
we have
\begin{equation}\label{bilinear_assumption}
a_m \langle \nabla^m u, \nabla^m u\rangle_x 
- 
\underline{a}\sum_{j=0}^{m-1} \langle \nabla^j u, \nabla^j u\rangle_x
\le 
\beta(u,u)_x
\le 
\overline{a} \sum_{j=0}^m\langle \nabla^j u, \nabla^j u\rangle_x .
\end{equation}
If we integrate over a region $\Omega\subset \M$, we obtain
$$
a_m |u|_{W_2^m(\Omega)}^2
 -
\underline{a}\|u\|_{W_2^{m-1}(\Omega)}^2
\le
[ u]_{m,\Omega}^2 := 
 \int_{\Omega}  \beta(u,u)_x \dif \mu (x) 
\le 
\overline{a} \|u\|_{W_2^m(\Omega)}^2.
$$

Now if $u$ vanishes on a sufficiently dense set $X_0\subset\Omega$,
then a corollary of the ``zeros'' estimate Theorem
\ref{omega_manf_continuous_est}, given in Section
\ref{lipschitz_domains}, will imply that
$
a_m |u|_{W_2^m(\Omega)}^2
 -
\underline{a} \|u\|_{W_2^{m-1}(\Omega)}^2
 \ge
\varepsilon \|u\|_{W_2^m(\Omega)}^2
$
where $\varepsilon$ depends on $a_m$, $\underline{a}$, properties of
$X_0$ and the boundary of $\Omega$, but nothing else.
The two most important types of subset $\Omega$, for our purposes, are
annuli $\a$ and complements
of balls $\b^{c}$. 

{\bf Annuli}. In this case we consider an annulus $\a =
B(p_0,R)\setminus B(p_0,R-t)$ with outer radius $R<\frac{1}{2}\inj$
and we apply Corollary \ref{zl_annulus} to a function $u$ that
vanishes on a set $X$ satisfying $h=h(X,\a) \le \gamma t$.  If, in
addition, $h\le \sqrt{\frac{a_m}{2\Lambda\overline{a}}}$, we have
$\|u\|_{W_2^{m-1}(\a)}^2\le \Lambda h^2 |u |_{W_2^m(\a)}^2\le
\frac{a_m}{2\underline{a}}|u|_{W_2^m(\a)}^2$ and, simultaneously,
$\frac{a_m}{4} \|u\|_{W_2^m(\a)}^2 \le \frac{a_m}{2}
|u|_{W_2^m(\a)}^2$, so $a_m |u|_{W_2^m(\Omega)}^2 - \underline{a}
\|u\|_{W_2^{m-1}(\Omega)}^2 \ge \frac{a_m}{4} \|u\|_{W_2^m(\Omega)}^2
$ follows and
\begin{equation}\label{annulus_comp_equivalence}
\frac{a_m}{4} \|u\|_{W_2^m(\a)}^2 
\le 
 \int_{\a}  \beta(u,u) \dif \mu \le \overline{a} \,\|u\|_{W_2^m(\a)}^2.
\end{equation}
(Note that $h$ must be chosen to be less than
$\sqrt{\frac{a_m}{2\Lambda\overline{a}}}$,
as well as $\gamma t$).

{\bf Complements of balls} We consider the punctured manifold $\b^c =
\M\setminus \b(p_0,R)$ with outer radius $R<\frac{1}{2}\inj$. We apply
Corollary \ref{zl_com} to a function $u$ that vanishes on a set $X$
satisfying $h=h(X,\a) \le \gamma t$.  By picking 
$h<\sqrt{  \frac{a_m}{2\Lambda\overline{a}}}$,
\begin{equation}\label{ball_comp_equivalence}
  \frac{a_m}{4} \|u\|_{W_2^m(\b(p,r)^c)}^2 
  \le  
  \int_{\b(p,r)^c}  \beta(u,u) \dif \mu \le \overline{a} 
  \|u\|_{W_2^m(\b(p,r)^c)}^2.
\end{equation}
follows. (Note that in this case, $h$ must be less than $h_0$ and
$\sqrt{ \frac{a_m}{2\Lambda\overline{a}}}$, but that it can be
chosen independently of $r$. In Lemma  \ref{simple_annulus} we 
refer to this critical value, the  minimum of $h_0$ and $\sqrt{ \frac{a_m}{2\Lambda\overline{a}}}$,
 as $H_0$.)

\section{The Lagrange function}
\label{lagrange_function}
We wish to uniformly bound the Lagrange function $\chi_\xi(x)$ and
establish its rate of decay as $x$ moves away from its center
$\xi$. There are two cases that we will consider. 

The first is the special case that involves interpolation by a
polyharmonic kernel $k_m$ (cf. Definition \ref{polyharmonic_kernels})
that is conditionally positive definite with respect to a space $\Pi$
annihilated by the operator $\L_m$. This case is significant because
the rate of decay is exponential (cf. Theorem~\ref{simple_decay}). It
includes the restricted surface splines on $\sph^d$ discussed in
Example \ref{example_rss}, for $d$ even.

The second case is the general one, where we do not assume any
annihilation properties concerning the space $\Pi$ that is to be
reproduced. This case includes the surface splines in odd
dimensions. The decay rate in this case is algebraic, rather than
exponential.

These results are similar to ones for the case of a lattice in $\RR^d$
\cite{buhmann_1990}. The restricted surface splines defined in
(\ref{def_rss}) have Lagrange functions that decay exponentially, for
$d$ even, but only algebraically for $d$ odd. For $d$ odd, the lattice
case has an additional family of polyharmonic splines with exponential
decay. We conjecture that this exponential decay holds for
odd-dimensional spheres, and that we have obtained only algebraic
decay is simply an artifact of the proof.

\begin{table}[ht]
\begin{center}
\begin{tabular}{|c|c|c|}
\hline
 {\bf Notation	}
& {\bf Constant}		
&{\bf Introduce in ...}
\\
 \hline
$a_m$	
& (Positive) lead coefficient of the polynomial $Q(x)$	
& Lemma \ref{Lap_Cov} 	\\ 
\hline
$\overline{a}$& 
maximum coefficient of $Q(x)$		&
(\ref{bars}) \\ 
&(in absolute value)& \\
\hline
$\underline{a}$& 
maximum coefficient of $Q(x)-a_m x^m$		&
(\ref{bars}) \\
 \hline
$C_{Q}$& 
$\ell_1(\reals^{\J})$ norm of eigenvalues of $\L_m|_{\Pi_{\J}}$		&
(\ref{spectral_max}) \\ 
\hline
$\inj$&
injectivity radius	&
Section \ref{background}
\\
\hline
$\Gamma_1,\Gamma_2$&
constants of metric equivalence from $\Exp$	&
(\ref{isometry})
\\
\hline
$c_1,c_2$&
constants of metric equivalence for Sobolev spaces
&
Lemma \ref{Fran}\\
\hline
$\Lambda$&
 constant for zeros lemma for annuli	&
(\ref{zeros_const})	\\ 
\hline
$h_0$ & threshold $h$ level for the zeros lemma & (\ref{common_h_0})
\\
\hline
$H_0$ & threshold $h$ level for results of \ref{ss_simplified} & Lemma \ref{simple_annulus}
\\
\hline
$H_1$ & 
threshold $h$ level for results of \ref{ss_general}
& 
Lemma \ref{general_annulus}
\\
\hline
\end{tabular}
\caption{Constants frequently used in Section \ref{lagrange_function}. The first
four constants are related to the elliptic operator $\L_m = Q(\Delta)$. The final
seven are geometric constants depending on $\M$.}
\end{center}
\end{table}

\subsection{$\L_m$ annihilates $\Pi_\J$}\label{ss_simplified}
We first consider the special case where $k_m$ satisfies
(\ref{reproduction}), with an operator $\L_m = Q(\Delta)$ for which
$\widetilde{k}_m(j) = (Q(j))^{-1}>0$ for $j\notin \J$ and $\L_m\phi_j
= 0$ for $j\in \J$. In other words, $k_m$ is conditionally positive
definite with respect to $\Pi_{\J}$, and $\L_m \Pi_{\J} ={0}$. This
the case for Example \ref{example_rss} for surface splines on even
dimensional spheres.

In this case, the native space seminorm (\ref{NS_norm}) is precisely 
the quadratic form derived from the operator, namely
$$
\ns{ u}_{k_m,\J}^2 
= 
\langle \L_m u, u\rangle_{L_2(\M)}
= 
[u]_{k_m,\M}^2.
$$ 
The more general case is considered in the next section, although the basic 
elements are present here.

We begin by observing that if $\Xi$ is sufficiently dense, with $h\le
\min(h_0, \sqrt{\frac{a_m}{2\Lambda \overline{a}}})$,
then by (\ref{ball_comp_equivalence}) it is possible to estimate the
norm of the Lagrange function by comparing it to a bump $\phi_{\xi}$
with $\phi_{\xi}\circ \Exp_{\xi} (x)= \sigma(|x|/q)$.  We note that
this bump is $1$ at $\xi$ and vanishes on the rest of $\Xi$, thus it
interpolates $\chi_{\xi}$ on $\Xi$ and has a smaller native space
seminorm.
\begin{equation}\label{First_Lagrange_Bound}
\frac{a_m}{4}\|\chi_{\xi}\|_{W_2^m(\M)}^2 
\le 
\ns{\chi_{\xi}}_{k_m,\J}^2
\le 
\ns{\phi_{\xi}}_{k_m,\J}^2
\le 
\overline{a}  \|\phi_{\xi}\|_{W_2^m(\M)}^2 
\le 
C \overline{a} q^{d-2m}.
\end{equation}
The final inequality follows from Lemma \ref{Fran}, and a direct computation of
$\| \sigma(|\cdot|/q)\|_{W_2^{2m}(\reals^d)}.$

The main result, the near-exponential decay of the Lagrange functions,
now is a consequence of an argument developed in \cite{HNW} but given
here in a somewhat different, streamlined form.  First we prove a
lemma showing that a fraction of the seminorm of the Lagrange function
$\chi_{\xi}$ taken over the punctured manifold $\b(\xi,r)^{c}$ resides
in a narrow annular region around the circle $\d(x,\xi) = r$.

\begin{lemma}\label{simple_annulus}
  Suppose $\M$ is a $d$-dimensional compact Riemannian manifold
  satisfying Assumption \ref{Manifold_Assumption}. Suppose further
  that $m>d/2$, $k_m$ satisfies Definition \ref{polyharmonic_kernels}
  and that $\L_m$ annihilates the space $\Pi_{\J}$. Then there is a
  constant $K>0$, depending only on $m$ and $\M$ so that the following
  holds.  If $\Xi$ is sufficiently dense, meaning that
$$
h<H_0:=\min \left(h_0, \sqrt{ \frac{a_m}{2\Lambda\overline{a}}} \right)
$$
and if $\aa = \b(p,r)\setminus \b(p,r-t)$ is an annulus of outer
radius $r<\inj$ and sufficient width $t$, so that 
$4h/h_0\le t$,
then the Lagrange functions for interpolation by $k_m$ satisfy
$$
\|\chi_{\xi}\|_{W_2^m(\b(p,r-t)^c)}^2 \le
K\|\chi_{\xi}\|_{W_2^m(\b(p,r)\setminus \b(p,r-t))}^2.
$$
\end{lemma}
%
\begin{proof}
Since $\chi_{\xi}$ minimizes the native space seminorm we have
$
[\chi_{\xi}]_{k_m,\M}^2
\le  
[\phi_{\xi} \chi_{\xi}]_{k_m,\M}^2
$ for any function $\phi_{\xi}$ equaling $1$ at $\xi$. If $\phi_{\xi}$
is a $C^{\infty}$ cut-off, equaling $1$ in the ball $\bb=\b(p,r-t)$
and vanishing outside of the ball $\bb\cup \aa$, then
$$
[\chi_{\xi}]_{k_m,\bb}^2 + [\chi_{\xi}]_{k_m,\bb^c}^2
\le  
[ \chi_{\xi}]_{k_m,\bb}^2
+ [\phi_{\xi} \chi_{\xi}]_{k_m,\aa}^2
$$
By (\ref{ball_comp_equivalence}) and (\ref{annulus_comp_equivalence})
$$ 
\frac{a_m}{4}\| \chi_{\xi}\|_{W_2^m(\bb^c)}^2 \le [\chi_{\xi}]_{k_m,\bb^c}^2\le  [\phi_{\xi} \chi_{\xi}]_{k_m,\aa}^2 \le \overline{a} \| \phi_{\xi} \chi_{\xi}\|_{W_2^m(\aa)}^2.
$$
The result follows with $K=4\overline{a} K'/a_m$, where the constant $K'$ is introduced in Lemma \ref{product_rule}, which we prove below.
\end{proof}
\begin{lemma}\label{product_rule}
Assume the manifold $\M$, the kernel $k_m$,  the set of centers $\Xi$ and the annulus $\aa\subset \M$ satisfy the conditions of Lemma \ref{simple_annulus}.
If $\phi_{\xi} $ is a smooth ``bump'' function, satisfying 
\begin{equation}\label{E:cutoff}
\phi_{\xi}\circ \mathrm{Exp}_{\xi} (x)
= 
\sigma\left(\frac{1}{t}\d(\Exp_{\xi}(x),\Exp_{\xi}(0)) -\frac{2t-r}{t}\right) 
= 
\sigma\left(\frac{|x |}{t}+\frac{2t-r}{t} \right)
\end{equation}
with $\sigma:\reals_+\to\reals_+$ a $C^{\infty}$, non-increasing cutoff function equaling $1$ on $[0,1]$and
$0$ on $[2,\infty)$,
then
$$\| \phi_{\xi} \chi_{\xi}\|_{W_2^m(\aa)}\le K' \|\chi_{\xi}\|_{W_2^m(\aa)}$$
where $K'$ depends only on $\M$, $m$ and the choice of cutoff $\sigma$.
\end{lemma}
%
\begin{proof}
We follow the proof of \cite[Lemma 4.3]{HNW}.
Let $ \widetilde{\chi_{\xi}}(x)  = \chi_{\xi} \circ \Exp_{\xi}$.
By using the metric equivalence guaranteed by Lemma \ref{Fran}, 
we can estimate $\| \phi_{\xi} \chi_{\xi}\|_{W_2^m(\aa)}^2$ by
\begin{eqnarray*}
\| \phi_{\xi} \chi_{\xi}\|_{W_2^m(\aa)}^2&\le& 
c_2^2 
\int_{\reals^d} 
\sum_{|\alpha|\le m} 
   \left|D^{\alpha} 
       \left[ \sigma\left(\frac{|x |}{t}+\frac{2t-r}{t} \right) \widetilde{\chi_{\xi}}(x)\right]
     \right|^2\dif x\nonumber\\
&\le& 
c_2^2C \sum_{j=0}^{m} t^{2(j-m)}
\int_{B(0,r)\setminus B(0,r-t)}
  \sum_{|\alpha|=j} 
  \left|D^{\alpha} 
       \widetilde{\chi_{\xi}}(x)
     \right|^2\dif x\nonumber\\
&\le& \left(\frac{c_2}{c_1}\right)^2C \sum_{j=0}^{m} t^{2(j-m)} \|\chi_{\xi}\|_{W_2^j(\aa)}^2
\le 
C  \left(\frac{c_2}{c_1}\right)^2 \Lambda \sum_{j=0}^m \left(\frac{h}{t}\right)^{2(m-j)} \|\chi_{\xi}\|_{W_2^m(\aa)}^2.\nonumber
\end{eqnarray*}
and $K' = C  (\frac{c_2}{c_1})^2 \Lambda \sum_{j=0}^m (h_0/4)^{2(m-j)} $. The second inequality follows
from the product rule, and $C$ is a constant depending only on $m$, $d$ and $\sigma$. The third inequality
is Lemma \ref{Fran} again, and the final inequality is the zeros lemma for annuli, Corollary \ref{zl_annulus}.
\end{proof}

At this point, we can follow the example of \cite[Section 4]{HNW}
\begin{theorem} \label{simple_decay} Suppose that $\M$ is a compact
  $d$-dimensional Riemannian manifold satisfying Assumption
  \ref{Manifold_Assumption}.  Suppose further that $m>d/2$ and that
  $k_m$ satisfies Definition \ref{polyharmonic_kernels} and that
  $\L_m$ annihilates the space $\Pi_{\J}$. There exist positive
  constants $h_0$, $\nu$ and $C$, depending only on $m$, $\M$ and the
  operator $\L_m$ so that if the set of centers $\Xi$ is quasiuniform
  with mesh ratio $\rho$ and has density $h\le H_0$ then the Lagrange
  functions for interpolation by $k_m$ satisfy
\begin{equation}\label{first_pointwise_lagrange_bound}
  |\chi_{\xi}(x)| \le C \rho^{m-d/2} 
  \exp\left(- \frac{\nu}{h} \min\bigl(d(x,\xi),\inj\bigr)\right).
 \end{equation}
 Furthermore, for any $0<\epsilon\le 1$, there is a constant $C$
 depending only on $m, \M, \rho$ and $\epsilon$, so that the Lagrange
 functions satisfy
\begin{equation}\label{first_holder_bound}
 |\chi_{\xi}(x)- \chi_{\xi}(y)| \le C\left(\frac{d(x,y)}{q}\right)^{\epsilon}.
 \end{equation}
\end{theorem}
\begin{proof}
  Set $t =4h/ h_0=: \gamma h$, and note that for $t\le r\le \inj$, Lemma
  \ref{simple_annulus} implies that
$$
\| \chi_{\xi}\|_{W_2^m(\b(\xi,r)^{c})}^2 \le 
\epsilon \| \chi_{\xi}\|_{W_2^m(\b(\xi,r-t)^{c})}^2, 
$$
with $\epsilon = (K-1)/K$. Letting $n := \lfloor r/t\rfloor$, we have
\begin{eqnarray*}
\| \chi_{\xi}\|_{W_2^m(\b(\xi,r)^{c})}^2 
&\le &
\epsilon^n \| \chi_{\xi}\|_{W_2^m(\M)}^2\\
&\le& 
\epsilon^{-1} e^{(\log \epsilon) r/t}\| \chi_{\xi}\|_{W_2^m(\M)}^2
\le 
\epsilon^{-1} e^{- \nu r/h}\| \chi_{\xi}\|_{W_2^m(\M)}^2\\
&\le& 
C e^{- \nu r/h}  q^{2m-d}
\end{eqnarray*}
with $\nu := - \gamma \log \epsilon$. Since $\epsilon = \frac{K}{K+1}<1$, 
it follows that $\nu>0$. 
The final inequality follows from (\ref{First_Lagrange_Bound}).

The bound (\ref{first_pointwise_lagrange_bound}) follows from the
observation that $\chi_{\xi}(x)$ can be estimated using
Theorem~\ref{omega_manf_continuous_est}. The intersection $\b(x,t)\cap
\b(\xi,\d(x,\xi))^{c}$ is contained the ball $\b(x,R)$, since
$t<R<\frac12 \inj$. Because geodesic spheres are smooth hypersurfaces
whose intersection is nontangential, The intersection $\b(x,t)\cap
\b(\xi,\d(x,\xi))^{c}$ is a Lipschitz domain contained in
$\b(x,R)$. Moreover, $h\le \gamma t$. Thus,
Theorem~\ref{omega_manf_continuous_est} applies, giving us
$$
|\chi_{\xi}(x)|
\le
C h^{m-d/2} \|\chi_{\xi}\|_{W_2^m(\b(x,t)\cap \b(\xi,\d(x,\xi))^{c})}
\le
C h^{m-d/2}\|\chi_{\xi}\|_{W_2^m\bigl(\b\bigl(\xi,d(x,\xi)\bigr)^c\bigr)}
$$
for $h<\gamma t$. Similarly, the estimate in
(\ref{first_holder_bound}) follows from Corollary \ref{he_balls}.
\end{proof}
\subsection{General Case}\label{ss_general}
In this case, the native space seminorm (\ref{NS_norm}) and the
quadratic form induced by the operator, differ by some low order
terms:
$$
\ns{ u}_{k_m,\J}^2 = \langle \L_m u, u\rangle_{L_2(\M)} -
\sum_{j\in\J} Q(\lambda_j) |\langle u, \varphi_j\rangle|^2 =
[u]_{k_m,\M}^2- \sum_{j\in\J} Q(\lambda_j) |\langle u,
\varphi_j\rangle|^2.
$$
Because of the orthonormality of $\varphi_j$, we have 
$ |\langle u, \varphi_j\rangle|^2\le \|u\|_{L_2(\M)}^2$.
Setting 
\begin{equation}\label{spectral_max}
C_Q := \sum_{j\in\J} |Q(\lambda_j)|
\end{equation}
(this is the $\ell_1(\reals^{\J})$ norm of the spectrum of the operator $\L_m$ restricted
to $\Pi_{\J}$) we note that, by Corollary~\ref{omega_full_manf_continuous_est}, if
$u$ vanishes on a sufficiently dense set, then the lower order terms 
are controlled
$$ \sum_{j\in\J}
|Q(\lambda_j)| |\langle u, \varphi_j\rangle|^2 \le C_Q h^{2m}
\|u\|_{W_2^m(\M)}^2.$$ Indeed it follows that $ \frac{a_m}{8}
\|u\|_{W_2^m}^2\le \ns{u}_{k_m,\J}^2$ when $h$ is chosen small
enough that $C_Q h^{2m}\le \frac{a_m}{8}$.

This allows us to provide a basic estimate for the Lagrange function,
similar to (\ref{First_Lagrange_Bound}). In this case
\begin{equation}\label{General_Lagrange_Bound}
\frac{a_m}{8}\|\chi_{\xi}\|_{W_2^m(\M)}^2
\le 
\ns{\chi_{\xi}}_{k_m,\J}^2
\le 
\ns{\phi_{\xi}}_{k_m,\J}^2
\le 
(\overline{a} + C_Q h^{2m})\|\phi_{\xi}\|_{W_2^m(\M)} 
\le 
C \overline{a} q^{d-2m}.
\end{equation}

\begin{lemma}\label{general_annulus}
  Suppose $\M$ is a $d$-dimensional compact Riemmanian manifold
  satisfying Assumption \ref{Manifold_Assumption}. Suppose further
  that $m>d/2$ and that $k_m$ satisfies Definition
  \ref{polyharmonic_kernels}. Then there is a constant $K>0$,
  depending only on $m$ and $\M$ so that the following holds.  If
  $\Xi$ is sufficiently dense, meaning that
$$
h< H_1 :=\min \left(H_0,  \sqrt[2m]{\frac{a_m}{8C_Q}}\right) =
\min \left(h_0, \sqrt{\frac{a_m}{2\Lambda \overline{a}} },
  \sqrt[2m]{\frac{a_m}{8C_Q}} \right)
$$
and if $\aa = \b(p,r)\setminus \b(p,r-t)$ is an annulus of outer
radius $r<\inj$, satisfying, in addition,
$$
\|\chi_{\xi}\|_{W_2^m(\b(p,r)^c)} \ge C_0 h^{2m} \|\chi_{\xi}\|_{W_2^m(\M)}
$$
(for a constant $C_0$ depending only on m, $\M$, $k_m$, and $\J$ which
we define in the proof) and sufficient width $t$, so that $h\le \gamma
t$, then the Lagrange functions for interpolation by $k_m$ with
auxiliary space $\Pi_{\J}$ satisfy
$$
\|\chi_{\xi}\|_{W_2^m(\b(p,r-t)^c)}^2 \le
K\|\chi_{\xi}\|_{W_2^m(\b(p,r)\setminus \b(p,r-t))}^2.
$$
\end{lemma}
\begin{proof}
Since $\chi_{\xi}$ minimizes the native space seminorm, we have
$$
[\chi_{\xi}]_{k_m,\M}^2 \le [\phi_{\xi} \chi_{\xi}]_{k_m,\M}^2 - \sum
Q(\lambda_j) \left( \left|\int_{\M}\phi_{\xi}(x) \chi_{\xi}
    (x)\overline{\varphi_{j}(x)}\dif \mu(x)\right|^2 - \left|
    \int_{\M}\chi_{\xi} (x)\overline{\varphi_{j}(x)}\dif
    \mu(x)\right|^2 \right)
$$ 
for a cut-off $\phi_{\xi}$ equaling $1$ in the ball $\bb=\b(p,r-t)$
and vanishing outside of the ball $\bb(p,r)=\bb\cup \aa$. Using the
sum of squares factorization $|A|^2 - |B|^2 =
\Re\bigl[(A-B)(\overline{A}+\overline{B})\bigr]$, we may write
\begin{eqnarray*}
 &\mbox{}& \left|\int_{\M} \phi_{\xi}(x) \chi_{\xi}(x) \overline{\varphi_j(x)}
    \dif \mu (x) \right|^2 -
  \left|\int_{\M} \chi_{\xi}(x) \overline{\varphi_j(x)} \dif \mu (x) \right|^2 \\
   &\mbox{}& \quad = \Re \left[ \left(\int_{\M} \bigl(\phi_{\xi}(x)-1\bigr)
      \chi_{\xi}(x) \overline{\varphi_j(x)} \dif \mu (x)\right) \times
    \left({\int_{\M}\bigl(\phi_{\xi}(x)+1\bigr)
      }\chi_{\xi}(x)\varphi_j(x) \dif \mu (x)\right)
  \right]\\
 &\mbox{}&\quad  = \Re \left[ \left( \int_{\aa} \bigl(\phi_{\xi}(x)\bigr)
      \chi_{\xi}(x) \overline{\varphi_j(x)} \dif \mu (x) -
      \int_{\bb^c} \chi_{\xi}(x) \overline{\varphi_j(x)} \dif \mu (x)\right)\right.\\
 &\mbox{}& \quad \quad    \times\left. \left({\int_{\M}\bigl(\phi_{\xi}(x)+1\bigr)
      }\chi_{\xi}(x)\varphi_j(x) \dif \mu (x)\right) \right].
\end{eqnarray*}
The second factor can be bounded by using
Corollary~\ref{omega_full_manf_continuous_est}, along with the cutoff
function $\phi_{\xi}$ being bounded by $1$ and $\|\varphi_j\|_2 =1$:
$$
\left|{\int_{\M}\bigl(\phi_{\xi}(x)+1\bigr) }\chi_{\xi}(x)\varphi_j(x)
  \dif \mu (x)\right| \le 2\|\chi_{\xi}\|_{L_2(\M)} \le 2\Lambda
h^m \|\chi_{\xi}\|_{W_2^{m}(\M)}
$$
To bound the first factor, start by using Corollary~\ref{zl_annulus}
and Lemma~\ref{product_rule} to obtain
\begin{eqnarray*}
\left|\int_{\aa} \bigl(\phi_{\xi}(x)\bigr) \chi_{\xi}(x)  \overline{\varphi_j(x)} \dif \mu (x)\right| 
  &\le& \left(\int_{\aa} |\phi_{\xi} \chi_{\xi}|^2\dif   \mu(x)\right)^{1/2}\\
  & \le& \Lambda h^m
       \|\phi_{\xi} \chi_{\xi}\|_{W_2^m(\aa)} 
\le \Lambda K' h^m \|
\chi_{\xi}\|_{W_2^m(\aa)}.
\end{eqnarray*}
Next, from Corollary~\ref{zl_com} we have that
\[
\left|\int_{\bb^c} \chi_{\xi}(x) \overline{\varphi_j(x)} \dif \mu
  (x)\right|\le \left(\int_{\bb^c} |\chi_{\xi}|^2\dif
      \mu(x)\right)^{1/2}\le
\Lambda h^m \|\chi_{\xi}\|_{W_2^m(\bb^c)}.
\]
So the first factor is bounded by 
\[
\Lambda K' h^m \|
\chi_{\xi}\|_{W_2^m(\aa)}+\Lambda h^m
\|\chi_{\xi}\|_{W_2^m(\bb^c)}\le\Lambda ( K'+1)h^m
\|\chi_{\xi}\|_{W_2^m(\bb^c)}, 
\]
and the product itself is bounded by 
\[
C'h^{2m} \|\chi_{\xi}\|_{W_2^m(\bb^c)} \|\chi_{\xi}\|_{W_2^{m}(\M)}, \
\text{where }C'=2\Lambda^2(K'+1).
\]
Putting these bounds together gives us
\begin{equation*}
\begin{split}
  \sum |Q(\lambda_j)|& \left| \left|\int_{\M}\phi_{\xi}(x) \chi_{\xi}
      (x) \overline{\varphi_{j}(x)}\dif \mu(x)\right|^2 - \left|
      \int_{\M}\chi_{\xi} (x)\overline{\varphi_{j}(x)}\dif
      \mu(x)\right|^2 \right| \\
  &\le C'C_Q h^{2m}
  \|\chi_{\xi}\|_{W_2^m(\M)} \|\chi_{\xi}\|_{W_2^m(\bb^c)}.
\end{split}
\end{equation*}

Thus for $h$ sufficiently small, say for $ C' C_Qh^{2m}
\|\chi_{\xi}\|_{W_2^m(\M)} \le \frac{a_m}{8}
\|\chi_{\xi}\|_{W_2^m(\bb^c)} $, which follows by taking
$$C_0 := \frac{8C' C_Q}{a_m},$$
we have
\begin{equation}\label{Qpart}
\sum |Q(\lambda_j)|
  \left|
    \left|\int_{\M}\phi_{\xi}(x) \chi_{\xi} (x)\overline{\varphi_{j}(x)}\dif \mu(x)\right|^2
    -
    \left| \int_{\M}\chi_{\xi} (x)\overline{\varphi_{j}(x)}\dif \mu(x)\right|^2
  \right|
\le \frac{a_m}{8}  \|\chi_{\xi}\|_{W_2^m(\bb^c)}^2.\end{equation}
We note from (\ref{ball_comp_equivalence}) that 
\begin{equation}
\label{tobesubtracted} 
  \frac{a_m}{4} \|\chi_{\xi}\|_{W_2^m( \bb^c) }^2 
\le   [\chi_{\xi}]_{k_m,W_2^m( \bb^c)}^2
\end{equation}
 and by subtracting right and left sides of (\ref{Qpart}) from the left and right sides of (\ref{tobesubtracted})  the lemma follows
since then 
\begin{multline*}
\frac{a_m}{8}\|\chi_{\xi}\|_{W_2^m( \bb^c)}^2 \\
\le
[\chi_{\xi}]_{k_m, \bb^c}^2
-
\left|
 \sum Q(\lambda_j)
   \left(
     \left|\int_{\M}\phi_{\xi}(x) \chi_{\xi} (x)\overline{\varphi_{j}(x)}\dif \mu(x)\right|^2
     -
     \left| \int_{\M}\chi_{\xi} (x)\overline{\varphi_{j}(x)}\dif \mu(x)\right|^2
   \right)
\right|\\
\le
[\phi_{\xi} \chi_{\xi}]_{k_m,\bb^c   }^2 
= [\phi_{\xi} \chi_{\xi}]_{k_m,\aa   }^2 
\le 
\overline{a}K'\|\chi_{\xi}\|_{W_2^m(\aa)}^2
\end{multline*}
where the last inequality follows from (\ref{annulus_comp_equivalence}). The result follows with $K = 8 \overline{a} K' /a_m$.
\end{proof}
%
%
We are now ready for the full result.
%
%
\begin{theorem} \label{general_decay}
Suppose that $\M$ is a compact $d$-dimensional Riemannian manifold satisfying Assumption \ref{Manifold_Assumption}. 
Suppose further that $m>d/2$ and that $k_m$ satisfies Definition \ref{polyharmonic_kernels}. 
There exist positive constants
$h_0$, $\nu$ and $C$, depending only on $m$, $\M$ and the operator $\L_m$ 
so that if the set of centers $\Xi$  is
quasiuniform with mesh ratio $\rho$ and has density $h\le H_1$
then the Lagrange functions for interpolation by $k_m$ with auxiliary space $\Pi_{\J}$ satisfy
\begin{equation}\label{general_pointwise_lagrange_bound}
 |\chi_{\xi}(x)| \le C \rho^{m-d/2} \max\left(\exp\left(- \frac{\nu}{h} \min\bigl(d(x,\xi),\inj\bigr)\right), h^{2m}\right).
 \end{equation}
Furthermore, for any $0<\epsilon\le 1$ for which $m>d/2$, there is a constant $C$ depending only on $m, \M, \rho$ and $\epsilon$, so that the  Lagrange functions satisfy
\begin{equation}\label{general_holder_bound}
 |\chi_{\xi}(x)- \chi_{\xi}(y)| \le C\left(\frac{d(x,y)}{q}\right)^{\epsilon}.
 \end{equation}
\end{theorem}
%
\begin{proof}
Let $r_0 $ be the smallest radius $r$ so that 
$\|\chi_{\xi}\|_{W_2^m(\b(p,r)^c)} \le C_0 h^{2m} \|\chi_{\xi}\|_{W_2^m(\M)}$.
Since $r\mapsto \|\chi_{\xi}\|_{W_2^m(\b(p,r)^c)}$ is decreasing, $r_0\le \text{diam}(\M)$.
Assume without loss that $r_0\le \inj$, since otherwise the proof proceeds exactly as
in Theorem \ref{simple_decay}.

Set $t =h/ \gamma$, and note that for $t\le r\le  r_0$, 
Lemma \ref{general_annulus} implies that 
$$\| \chi_{\xi}\|_{W_2^m(\b(\xi,r)^{c})}^2 \le \epsilon \| \chi_{\xi}\|_{W_2^m(\b(\xi,r-t)^{c})}^2$$
with $\epsilon = (K-1)/K$. 

As in the proof of Theorem \ref{simple_decay}
$$
\| \chi_{\xi}\|_{W_2^m(\b(\xi,r)^{c})}^2 
\le 
\epsilon^{-1} e^{- \nu r/h}\| \chi_{\xi}\|_{W_2^m(\M)}^2
\le 
C e^{- \nu r/h}  q^{2m-d}.
$$
Where we have set $\nu := - \gamma \log \epsilon$. Since $\epsilon = \frac{K}{K+1}<1$, 
it follows that $\nu>0$. 
The last inequality follows from (\ref{General_Lagrange_Bound}). 
On the other hand, for $r\ge r_0$, we have that $\|\chi_{\xi}\|_{W_2^m(\b(p,r)^c)} \le C_0 h^{2m} \|\chi_{\xi}\|_{W_2^m(\M)} \le C C_0 q^{2m-d} h^{2m},$ by (\ref{General_Lagrange_Bound}).
Therefore,
$$\|\chi_{\xi}\|_{W_2^m(\b(p,r)^c)} \le C q^{2m-d}  \max(h^{2m}, e^{- \nu r/h} )$$

Again, estimate (\ref{general_pointwise_lagrange_bound}) follows from the observation that $\chi_{\xi}(x)$ can be estimated by way of the zeros lemma:
$$
|\chi_{\xi}(x)| \le 
C h^{m-d/2} \|\chi_{\xi}\|_{W_2^m(\b(x,t)\cap \b(\xi,\d(x,\xi))^{c})}
\le
C h^{m-d/2} \|\chi_{\xi}\|_{W_2^m\bigl(\b\bigl(\xi,d(x,\xi)\bigr)^c\bigr)},
$$ for $h<\gamma t$.

Similarly, estimate (\ref{general_holder_bound}) follows from
Corollary \ref{he_balls}.
\end{proof}

\subsection{Implications for Interpolation and Approximation}
At this point, we are able to state three important corollaries to
Theorem \ref{general_decay} that satisfactorily answer the questions
concerning bases and approximation properties of $V_X$ discussed in
Section~\ref{sec0}. These results were previously obtained in
\cite{HNW,HNSW} for a class of Sobolev kernels. Here, we get them for
a much broader, computationally implementable class of kernels. Our
first result is that the Lebesgue constant for interpolation is
uniformly bounded.

\begin{theorem}[Lebesgue Constant]\label{LC}
  Let $\M$ be a compact Riemannian manifold of dimension $d$
   satisfying Assumption \ref{Manifold_Assumption}.  Suppose further that
  $m>d/2$ and that $k_m$ satisfies Definition
  \ref{polyharmonic_kernels}.  For a quasi-uniform set $\Xi\subset \M$,
  with mesh ratio $h/q \le \rho $, if 
  $h\le H_0$, 
  then the
  Lebesgue constant, $L = \sup_{\alpha\in \M} \sum_{\xi\in \Xi}
  |\chi_{\xi}(\alpha)|$, associated with $k_{m}$ and
  $\J$ is bounded by a constant depending only on $m$, $\rho $ and $\M$.
\end{theorem}
\begin{proof}
Fix $x$.
Using 
Theorem \ref{general_decay}, we estimate the sum as
$$\sum_{\xi\in \Xi} |\chi_{\xi}(x)| \le
\sum_{\xi\in \Xi} C \rho^{m-d/2}\exp\left({-\nu \frac{\min(\d(x,\xi),\inj)}{h}}\right)
+\sum_{\xi\in \Xi} C \rho^{m-d/2} h^{2m} 
=:I+II
$$ 

The first sum can be treated exactly as in \cite[Theorem 4.6]{HNW}, and is
bounded independently of $h$.
The second sum, $II$, can be estimated  
using the fact that $\#\Xi \le C  q^{-d}$,
with a constant $C = \mu(\M)/ \alpha(\M)$, where $\alpha(\M) := \inf_{x\in\M}\inf_{0<r<\inj}r^{-d} \mu(\b(x,r))$.
Thus
$$II \le  C  \rho^{m-d/2}q^{-d} h^{2m}\le C  h^{2m-d} \rho^{m+d/2}
$$
which is bounded since $2m>d$ (indeed it vanishes as $h\to 0$).
\end{proof}
%

The next consequence is the $L_p$ stability of the Lagrange basis.
To this end, we define 
$$S(k_m,\J,\Xi)  := \left\{\sum_{\xi \in\Xi} A_{\xi} k_m(\cdot,\xi) + p \mid p\in \Pi_{\J}\ \text{and}\  \sum A_{\xi} q(\xi) = 0\ \text{for all}\ q\in\Pi_{\J}\right\}$$
and use this notation in lieu of $V_X$ used in the introduction.

\begin{theorem}[Stability of Lagrange Basis]\label{stability}
Under the assumptions of Theorem \ref{general_decay},
there exist constants $0<c_1<c_2$, depending only on $k_m$, $\J$, $\M$ and $\rho$ so that
$$
c_1\|A_{p,\cdot}\|_{\ell_p(\Xi)} 
\le 
\|s \|_{L_p(\M)}
\le 
c_2 \|A_{p,\cdot}\|_{\ell_p(\Xi)}.
$$
holds for all $s =\sum_{\xi \in \Xi} A_{\xi} \chi_{\xi}  \in S(k_m,\J,\Xi)$, with 
normalized coefficients $A_{p,\xi} := q^{d/p} A_{\xi}$.
\end{theorem}
\begin{proof}
When $\L_m$ annihilates $\Pi_{\J}$, this is a direct consequence of the pointwise estimates obtained in Theorem \ref{simple_decay}. 
We observe that the following  three conditions hold:
\begin{enumerate}
\item The basis $\left(\chi_{\xi}\right)_{\xi \in \Xi}$ is a Lagrange basis.
\item The basis has decay $|\chi_{\xi}(x)|\le C \rho^{m-d/2} \exp\left(- \frac{\nu}{h} \min\bigl(d(x,\xi),\inj\bigr)\right)$.
\item The basis has the equicontinuity condition 
$
| \chi_{\xi}(x) - \chi_{\xi}(y)|
\le 
C_2 \left[\frac{\d(x,y)}{q}\right]^{\epsilon}. 
$
\end{enumerate}
Thus the result \cite[Theorem 3.10]{HNSW} applies.

In the general case, the result still holds, despite the fact that item 2 may fail. I.e., the Lagrange functions may
decay more slowly than the basis functions considered in \cite{HNSW}, and a minor modification 
is required to apply of \cite[Theorem 3.10]{HNSW}.

The upper bound 
$\|s \|_{L_p(\M)}
\le 
c_2 \|A_{p,\cdot}\|_{\ell_p(\Xi)}$ 
follows directly from the estimate (\ref{general_pointwise_lagrange_bound}).
Indeed, the case $p=\infty$,  is none other than the Lebesgue constant estimate Theorem \ref{LC}, while 
the $p=1$ case follows by the uniform bound on 
$$\|\chi_{\xi}\|_1\le C\rho^{m-d/2}\left( C h^d + \mathrm{vol}(\M)\left(h^{2m} + e^{-\nu \inj /h}\right)\right)
 \le C \rho^{m+d/2}  q^d.$$ 
 The case $1<p<\infty$ follows by interpolation.

To handle the lower bound, we utilize functions $\phi_{\xi}$, 
defined in a similar way as  in (\ref{E:cutoff}), satisfying $\phi_{\xi}\circ \mathrm{Exp}_{\xi} = \sigma$,
with
\begin{equation}\label{basis_split}
\sigma(x)
= 
\begin{cases}
1 & |x|\le r_0\\
h^{-2m}e^{-\frac{\nu|x|}{h}} & r_0<|x| \le \inj\\
h^{-2m}e^{-\frac{\nu\inj}{h}} & |x|> \inj
\end{cases}
\qquad \text{or} 
\qquad 
\sigma(x)
= 
\begin{cases}
1 & |x|\le r_0\\
h^{-2m}e^{-\frac{\nu|x|}{h}} & r_0<|x| 
\end{cases}
\end{equation}
and with threshold value $r_0 := -\frac{2m}{\nu}h \log h$ (the second
definition is chosen if $\inj< r_0$).  
It follows that 
$$\chi_{\xi} =\chi_{\xi} \phi_{\xi} + \chi_{\xi} (1- \phi_{\xi}) =: g_{\xi} + b_{\xi},$$
and $g_{\xi}$ satisfies items 1--3 above 
(in particular, item 3 follows since $\phi_{\xi}$ is bounded and $\mathrm{Lip}(1)$, with Lipschitz
constant $\nu/h$)
and \cite[Theorem 3.10]{HNSW} applies. In particular, there is $c_1>0$ so that
$\|\sum_{\xi\in\Xi} A_{\xi} g_{\xi}\|_p\ge c_1\|A_{p,\cdot}\|_{\ell_p(\Xi)} $.

On the other hand, $|b_{\xi}(x)| \le h^{2m}$, implies that
$\|\sum_{\xi\in\Xi} A_{\xi} b_{\xi}\|_p \le C \rho^d \|A_{p,\cdot}\|_{\ell_p{\Xi}} h^{2m-d}$
since
$$\int_{\M} |\sum_{\xi\in\Xi} A_{\xi} b_{\xi}(x)|\dif x \le  \|A\|_{\ell_1(\Xi)} \mu(\M)h^{2m}
\quad \text{and} \quad
\max_{x \in \M} |\sum_{\xi\in\Xi} A_{\xi} b_{\xi}(x)|\le C  \rho^d \| A \|_{\ell_{\infty}(\Xi)}h^{2m-d}.$$
Thus, for
$s = \sum_{\xi\in \Xi} A_{\xi} \chi_{\xi}$,
$$\|s\|_p  = \|\sum_{\xi\in \Xi} A_{\xi} \chi_{\xi}\|_p\ge 
\left(2^{1-p} \|\sum_{\xi\in \Xi} A_{\xi} g_{\xi}\|_p^p -  
\|\sum_{\xi\in \Xi} A_{\xi} b_{\xi}\|_p^p\right)^{1/p}
\ge 
\left(\frac{c_1}{2}-o(h)\right)\, \|A_{p,\cdot}\|_{\ell_p(\Xi)},  $$
where we have used the inequality 
$|\sum_{\xi\in\Xi} A_{\xi} g_{\xi}|^p \le2^{p-1}\left(|\sum_{\xi\in\Xi} A_{\xi} \chi_{\xi}|^p +|-\sum_{\xi\in\Xi} A_{\xi} 
b_{\xi}|^p\right)$.
\end{proof}

Our final consequence treats the 
$L_p$ stability of the $L_2$ projector. This was a primary
goal of \cite{HNSW}, and, in light of Theorem \ref{general_decay}, we can
produce a similar result here, with a minor modification to handle
the slower decay of the Lagrange functions.

Let $V:\comps^{\Xi}\to S(k_m,\J,\Xi)$ be a basis
``synthesis operator'' $V: (A_{\xi})_{\xi\in\Xi} \mapsto \sum_{\xi\in\Xi} A_{\xi} v_{\xi}$,
for a basis $(v_{\xi})_{\xi \in \Xi}$ of $S(k_m,\J,\Xi)$.
Likewise, let $V^{*}:L_1(\M)\to \comps^{\Xi}$ be its formal adjoint 
$V^{*}:f \mapsto \bigl(\langle f, v_{\xi}\rangle\bigr)_{\xi\in\Xi}.$
The $L_2$ projector is then
 $T_{\Xi} = V (V^* V)^{-1} V^*: L_1(\M) \to S(k_m,\J,\Xi) ,$ in the sense
 that when $f\in L_2(\M)$,  $T_{\Xi} f$ is the best $L_2$ approximant to $f$ from $S(k_m,\J,\Xi)$.
 
 The $L_2$ norm of this
 projector is $1$ (it being an orthogonal projector), while the $L_p$ and $L_p'$ norms
 are equal, because it is self-adjoint. Thus, to estimate its $L_p$ operator norm ($1\le p\le \infty$),
 it suffices to estimate its $L_{\infty}$ norm.
 
\begin{theorem}\label{main}
Under the assumptions of Theorem \ref{general_decay},
for all
$1\le p\le \infty$, the $L_p$ operator norm of the $L_2$  projector 
$T_{\Xi}$ is
bounded by a constant depending only on $\M$, $\rho$ $k_m$ and $\J$.
\end{theorem}
\begin{proof} 
  When $\L_m$ annihilates $\Pi_{\J}$, Theorem \ref{simple_decay} and
  Theorem \ref{stability} satisfy the conditions of \cite[Theorem
  5.1]{HNSW} (the Lagrange basis is stable and rapidly decaying), and
  the result follows.

  In the general case, we cannot directly apply this theorem, because
  the basis does not decay rapidly enough.  We take as our basis
  $v_{\xi} = \chi_{\xi,2} := q^{-d/2}\chi_{\xi}$, the $L_2$ normalized
  Lagrange basis.  It follows from Theorem \ref{stability} that
  $\|V\|_{\ell_{\infty}(\Xi) \to L_{\infty}(\M)}\le c_2q^{-d/2}$ and
  $\|V^*\|_{L_{\infty}(\M) \to \ell_{\infty}(\Xi)}\le c_2q^{d/2}$.
  Thus, to estimate the $L_{\infty}$ operator norm of $T_{\Xi}$ (and
  thereby all other $L_p$ norms), it suffices to estimate the
  $\ell_{\infty}(\Xi)\to \ell_{\infty}(\Xi)$ norm of the inverse Gram
  matrix $(V^* V)^{-1}$.

We make the split $g_{\xi} =\chi_{\xi} \phi_{\xi}$ and $b_{\xi}= \chi_{\xi} (1- \phi_{\xi})$
 with $\phi_{\xi}\circ \mathrm{Exp}_{\xi} = \sigma$ defined as in (\ref{basis_split}). 
And note that
$\chi_{\xi,2} =\chi_{\xi,2} \phi_{\xi} + \chi_{\xi,2} (1- \phi_{\xi}) =: g_{\xi,2} + b_{\xi,2},$

It follows
 that $V^* V = G + B$, with $G_{\xi,\zeta} = \langle g_{\xi,2}, g_{\zeta,2}\rangle_{L_2(\M)}$.
 
 The functions $(g_{\xi})$ are a Lagrange basis, 
 in the sense that $g_{\xi}(\zeta) = \delta_{\xi,\zeta}$
 (although they span a different space than $S(k_m,\J,\Xi)$
 and, as observed in the proof of Theorem \ref{stability}, they are $L_p$ stable. They also satisfy the decay
 conditions of \cite[Proposition 4.1]{HNSW}, and by applying this result we see that 
 $\|G^{-1}\|_{\infty}$ is bounded by a constant.
 
 On the other hand, $|B_{\xi,\zeta}| \le  
 |\langle g_{\xi,2},b_{\zeta,2}\rangle| +
 |\langle b_{\xi,2},g_{\zeta,2}\rangle| +  |\langle b_{\xi,2},b_{\zeta,2}\rangle| 
 \le C h^{2m}$, and
  $\|B\|_{\infty} \le \max_{\xi} \sum_{\zeta} |B_{\xi,\zeta}|\le C \rho^d h^{2m-d}.$
  The theorem follows by noting that
  $(V^* V)  = G (\mathrm{Id}+ G^{-1} B)$, and, hence,
$\|(V^* V) ^{-1}\|_{\infty} \le \|G^{-1}\|_{\infty} \left(1+ o(h)\right).$

\end{proof}
\subsection{Spheres and $SO(3)$}
We now explore some further consequences of the results of the previous section. 
We will shortly see that Theorems \ref{LC} and \ref{main} imply that $I_{\Xi}$ and $T_{\Xi}$ are {\em near-best}. In some important cases,  we can then use these projectors to observe precise rates of convergence for interpolation
and least squares minimization, better rates than were previously known. 

For the kernels considered in section~\ref{ss_examples}, theoretical 
approximation results are known in some special cases, including spheres and $SO(3)$. The difficulty is that these results are often not practical, because they are derived from approximation schemes that are difficult to implement.
The good news is that the stability of the schemes $I_{\Xi}$ and $T_{\Xi}$ imply that these
operators, which are associated with \emph{practical} schemes, inherit the \emph{same} convergence rates. Indeed, for a normed linear space $Y$ and
a bounded projector $P:Y \to Y$, one has 
for $f\in Y$, 
\begin{equation}
\label{lebesgue_ineq}
\|f- Pf\| =  \inf_{s\in \mathrm{ran}P} \|f - s + Ps - Pf\| \le (1+\|P\|) \d(f,\mathrm{ran} P).
\end{equation}
This fundamental observation is known as a Lebesgue inequality,
and we employ it with $P=I_{\Xi}$ and $Y= C(\M)$ as well as with $P =  T_{\Xi}$ and  $Y= L_p(\M)$.
In recent years, a concerted effort has been undertaken\footnote{This stands in contrast
to the more classical, mainstream theory of kernel approximation, where approximation
properties are investigated and understood only for functions coming from the reproducing
kernel (semi-)Hilbert space associated with a conditionally positive definite kernel. Obtaining an understanding outside of this context has generally required
indirect, theoretical approximation schemes, and it has not been obvious, until now, 
that such results would have practical consequences.}
 to understand the general $L_p$
convergence rates (i.e., the behavior of $ \d(f,S(k_m, \J, \Xi))_p$  as $\Xi$ becomes dense in $\M$)
of certain well-known kernels in terms of smoothness assumptions on the target function $f$ 
and on the density of the point-set $\Xi$, measured
by the fill distance $h$. 

To measure smoothness of the target function, we make use of 
the classical (Sobolev, Besov) smoothness spaces introduced in Definition \ref{Sob_Norm} and Definition \ref{Besov},
with the exception that for approximation in $L_{\infty}$, we make the 
(usual) replacement of $C^{2m}$ for  $W_{\infty}^{2m}$ (but using the same norm).
As a shorthand, we capture the smoothness spaces we use by means of a common
notation, $\mathcal{W}_p^s(\M)$. For  $m>d/2$
denote the space $\mathcal{W}_p^s(\M)$  by
\begin{itemize}
\item  $\mathcal{W}_p^s(\M) = C^{2m}(\M)$, when $p= \infty$ and $s=2m$
\item $\mathcal{W}_p^s(\M) = W_p^{2m}(\M)$ when $1\le p< \infty$ and $s=2m$
\item  $\mathcal{W}_p^s(\M) = B_{p,\infty}^{s}(\M)$ when $1\le p\le \infty$ and $0<s<2m$.
\end{itemize}
\begin{corollary}\label{sphere_result}
For $\M=\sphere$, and $m>d/2$
the surface splines introduced in Example 3: $k_{m}(x,y) = \phi_s(x\cdot,y)$ satisfy the following.
There is a constant $C$ so that, for a sufficiently dense set $\Xi\subset \sphere$,  and for $f\in\mathcal{W}_p^s$ .
\begin{enumerate}
\item For $f\in \mathcal{W}_{\infty}^s$, $\|I_{\Xi} f - f\| \le C h^{s} \|f\|_{\mathcal{W}_{\infty}^s}$
\item For $1\le p\le \infty$ and for $f \in \mathcal{W}_p^s,$ $\|T_{\Xi} f - f\|_{p} \le h^{s} \|f\|_{\mathcal{W}_{\infty}^s}$.
\end{enumerate}
\end{corollary}
\begin{proof}
  As with the Sobolev kernels, $\phi_s$ is of the form $G_{\beta}
  +\psi*G_{\beta}$ as considered in \cite{MNPW}, the result follows
  direct from \cite[Theorem 6.8]{MNPW}. Alternatively, it follows from
  \cite[Theorem 6.1]{Hsphere}, which treats kernels on the sphere of
  the type in Definition \ref{polyharmonic_kernels}. The Besov space
  result follows from \cite[Corollary 6.13]{MNPW} or \cite[Corollary
  6.2]{Hsphere}.
\end{proof}

\begin{corollary}\label{SO(3)_result}
  For $\M=SO(3)$, and $m\ge 2$ The surface splines, $\bfk_{m}$,
  introduced in Example 4 satisfy the following.  There is a constant
  $C$ so that, for a sufficiently dense set $\Xi\subset SO(3)$, and
  for $f\in\mathcal{W}_p^s$ .
\begin{enumerate}
\item For $f\in \mathcal{W}_{\infty}^s$, $\|I_{\Xi} f - f\| \le C h^{s} \|f\|_{\mathcal{W}_{\infty}^s}$
\item For $1\le p\le \infty$ and for $f \in \mathcal{W}_p^s,$ $\|T_{\Xi} f - f\|_{p} \le h^{s} \|f\|_{\mathcal{W}_{\infty}^s}$.
\end{enumerate}
\end{corollary}
\begin{proof}
This follows from \cite[Theorem 9]{HS} for the case of full smoothness and
from \cite[Theorem 12]{HS} when $0<s<2m$.
\end{proof}
\begin{appendix}

  \section{A Zeros Lemma for Lipschitz Domains on
    Manifolds} \label{lipschitz_domains}
  
Results concerning Sobolev bounds on functions with many zeros are
known for Lipschitz domains in $\RR^d$
\cite{NWW,Narcowich-etal-06-1}. Our aim is to extend these results to
certain Lipschitz domains on manifolds. Before we do that, however, we
will need to improve the $\RR^d$ results in
\cite{NWW,Narcowich-etal-06-1}.

\subsection{Lipschitz domains in $\RR^d$}
Consider a domain $\Omega\subset \reals^d$ that is bounded, has a
Lipschitz boundary, and satisfies an interior cone condition, where
the cone $C_\Omega$ has a maximum radius $R_0$ and aperture\footnote{Aperture here is the angle across the cone, $2\varphi$ in this case. In optics, aperture would be $\varphi$. }
$2\varphi$. Of course, the cone condition will be obeyed if we use any
radius $0<R\le R_0$. The theorem that we will give below requires
covering $\Omega$ with certain star-shaped domains.

We will say that a domain $\stardom$ is \emph{star shaped with respect
  to a ball} $B(x_c,\inrad):= \{x\in \RR^d\colon |x-x_c|<\inrad \}$
if, for every $x\in \stardom$, the closed convex hull of $\{x\}\cup B$
is contained in $\stardom$ \cite[Chapter 4]{Brenner-Scott-94-1}.  For
$\stardom$ bounded, there is a measure of how close to spherical
$\stardom$ is; namely, the \emph{chunkiness parameter} $\gamma$
\cite[Definition 4.2.16]{Brenner-Scott-94-1}.  This is defined as the
ratio of $d_\stardom$ to the radius of the largest ball relative to
which $\stardom$ is star shaped. When $\stardom$ is a sphere,
$\gamma=2$. If there is a ball $B(x_c,R) \supseteq \stardom$, then
$r<d_\stardom <2R$ and $\gamma \le \frac{2R}{\inrad}$. Finally, such
domains satisfy an interior cone condition and certain Sobolev bounds,
which are stated in the next two propositions.

\begin{proposition}[{\cite[Proposition 2.1]{NWW}}]
\label{cone_cond}  
If $\stardom$ is bounded, star shaped with respect to $B(x_c,\inrad)$
and contained in $B(x_c,R)$, then every $x\in \stardom$ is the vertex
of a cone $C_\stardom \subset \stardom$ having radius $\inrad$ and aperture
$\,\theta :=2\arcsin\left(\frac{\inrad}{2R}\right)$. 
\end{proposition} 

This proposition also implies that the chunkiness parameter for
$\stardom$ is bounded in terms of the aperture:
\[
\gamma \le \frac{2R}{r} = \csc(\theta/2).
\]

\begin{proposition}[{\cite[Proposition 3.5]{HNW}}]
\label{stardom_est}  
  Let $\stardom\subset \reals^d$ be as above,
  $m\in \nats$ and $\sfp\in \reals$, $1\le \sfp \le \infty$. Assume
  $m>d/\sfp$ when $\sfp>1$, and $m\ge d$, for $\sfp=1$. If $u\in
  W_\sfp^m(\stardom)$ satisfies $u|_X=0$, where
  $X=\{x_1,\ldots,x_N\}\subset \stardom$ and if $h=h_X\le
  \frac{d_\stardom}{16m^2\gamma^2}$, then
\begin{equation}
\label{p_bound_W_k_u}
|u|_{W_\sfp^k(\stardom)} \le C_{m,d,\sfp} \gamma^{d+2k} d_\stardom^{m-k}  
|u|_{W_\sfp^m(\stardom)}
\end{equation}
\begin{equation}
\label{infty_bnd_u}
\|u\|_{L_\infty(\stardom)} \le C_{m,d,\sfp} \gamma^{d} 
d_\stardom^{m-d/\sfp} |u|_{W_\sfp^{m}(\stardom)}.
\end{equation} 

\end{proposition}

Our next task is to obtain Sobolev bounds for the domain
$\Omega\subset \RR^d$ that are similar those in
(\ref{p_bound_W_k_u}). The idea is to cover $\Omega$ with star-shaped
domains. To do that, we will use a construction due to Duchon
\cite{D2}. With $R_0$, $2\varphi$ being the radius and aperture for the
cone $C_\Omega$, and $0<R\le R_0$, let
\begin{equation}  
\label{defs_r_Tr}  
r := 2RF(\varphi),\ \text{where }F(\varphi) 
:=\frac{\sin(\varphi)}{4(1+ \sin(\varphi))}, \ \mbox{and }  
T_r :=\left\{t\in \frac{2r}{\sqrt{d}}\ints^d \colon  
  B(t,r) \subset \Omega \right\}.
\end{equation}  
For $t\in T_r$, let $\stardom_t$ be the set of all $x\in \Omega$ such
that the closed convex hull of $\{x\}\cup B(t,r)$ is contained in
$\Omega\cap B(t,R)$. From \cite[Lemma~2.11]{NWW}, we have that each
$\stardom_t$ is star shaped with respect to the ball $B(t,r)$, and
satisfies $B(t,r)\subseteq \stardom_t\subseteq \Omega \cap B(t,R)$,
$d_{\stardom_t}<2R$. Because $2R/r = 1/F(\varphi)$, the aperture for
$C_{\stardom_t}$ is
\[
\theta=2\arcsin(1/F(\varphi)),
\]
and the chunkiness parameter $\gamma_t$ for $\stardom_t$ is uniformly
bounded:
\begin{equation}
\label{gamma_t_bnd}
2\le \gamma_t < \frac{2R}{r} = \frac{1}{F(\varphi)}.
\end{equation}
We also have that $\Omega = \bigcup_{t\in T_r}\stardom_t$, that
$\#T_r< C_{d}\,\mathrm{vol}(\Omega) (F(\varphi)R)^{-d}$, and that
$\mathrm{vol}(\stardom_t)\le C_d R^d$.

The integer-valued simple function $\sum_{t\in T_r}\chi_{B(t,R)}(x)$
is the number of $B(t,R)$'s that contain $x$. This is easily bounded
above by $ M_{d,\varphi}$, maximum number of such balls intersecting a
fixed one, say $B(0,R)$. A little geometry shows that
\[
M(d,\varphi) \le (2R/r+1)^d\le 2^d/(F(\varphi)^d
\] 
Note that the existence of $M_{d,\varphi}$ implies that for any
function $f$ in $L_1(\Omega)$ we have
 \begin{equation}
 \label{sum_t_bound}
 \sum_t \int_{\stardom_t}|f(x)|dx = \int_\Omega \sum_t
 \chi_{\stardom_t}(x) 
 |f(x)|dx \le M_{d,\varphi}\int_\Omega |f(x)|dx \le (
 2^d/(F(\varphi)^d) 
 \int_\Omega |f(x)|dx.
\end{equation}

\begin{lemma}\label{uniform_t}
  Suppose that $h=h_{X,\Omega}$ satisfies $h\le \frac{R}{8
    m^2}F(\varphi)^2$, then (\ref{p_bound_W_k_u}) and
  (\ref{infty_bnd_u}) hold uniformly in $t$ for $\stardom_t$, provided
  $\gamma_t$ and $d_{\stardom_t}$ are replaced by $1/F(\varphi)$ and
  $2R$, respectively.
\end{lemma}

\begin{proof}
The mesh norm for $\Omega$ satisfies 
\[
h \le m^{-2}(\underbrace{2RF(\varphi)}_r)\times (F(\varphi)/16)<r,
\]
since $F(\varphi)<1$. It follows that $B(x_c,r)\cap X\neq \emptyset$,
and so $\stardom_t\cap X$ contains at least one point of $X$. From
this we have that $h_{\stardom_t\cap X} \le h$. The lemma then follows
from the bound on $h$ being less than the one required in
Proposition~\ref{stardom_est}.
\end{proof}

We wish to prove the following result, which differs from an earlier
result in \cite[Theorem~2.12]{NWW} in that it applies to cases in
which the index $k\le m-1$, as opposed to $k< m-n/p$. 

\begin{theorem}[{\bf Euclidean Case}]
  \label{omega_euclidean_lip}
  Suppose that $\Omega$ is a Lipschitz domain obeying a cone
  condition, where the cone $C_\Omega$ has radius $R_0$ and aperture
  $2\varphi$.  Let $k$, $m$, and $p$ be as in
  Proposition~\ref{stardom_est}, and and let $X\subset \Omega $ be a
  discrete set with mesh norm $h$ satisfying
\begin{equation}
\label{h_cond_Omega}
h< \frac{R_0}{8 m^2}F(\varphi)^2.
\end{equation}
If $u\in W_p^m(\Omega)$ satisfies $u|_X=0$,
then
\begin{equation}
\label{p_bound_W_k_u_omega}
|u|_{W_\sfp^k(\Omega)} \le \frac{2^{d/\sfp}(4m)^{2m-2k}
  C_{m,d,\sfp}}{F(\varphi)^{2m+d+d/\sfp}}
h^{m-k}|u|_{W_\sfp^m(\Omega)} 
\end{equation}
and
\begin{equation}\label{infty_bound_u_omega}
\|u\|_{L_\infty(\Omega)}
\le \frac{(4m)^{2m-2d/\sfp}
  C_{m,d,\sfp}}{F(\varphi)^{2m+d-2d/\sfp}}
h^{m-d/\sfp}|u|_{W_\sfp^m(\Omega)}.
\end{equation} 
\end{theorem}  
  
\begin{proof}
  Given $h$, choose $R = 8 m^2h/F(\varphi)^2<R_0$.  Applying
  Lemma~\ref{uniform_t} and Proposition~\ref{stardom_est} to the
  domain $\stardom_t$ then results in the bound
\[
|u|_{W_\sfp^k(\stardom_t)} \le
\frac{(4m)^{2m-2k}C_{m,d,\sfp}}{F(\varphi)^{d+2m}} h^{m-k}
|u|_{W_\sfp^m(\stardom_t)}
\]
We will follow the proof in \cite[Theorem~2.12]{NWW}. Summing over $t$
on both sides of the previous inequality, using $\Omega=\cup_t
\stardom_t$ and applying (\ref{sum_t_bound}), we have that
\[
|u|^\sfp_{W_\sfp^k(\Omega)} \le
\left(\frac{(4k)^{2m-2k}C_{m,d,\sfp}h^{m-k}}{F(\varphi)^{d+2m}}\right)^\sfp
( 2^d/(F(\varphi)^d) |u|^\sfp_{W_\sfp^m(\Omega)},
\]
from which (\ref{p_bound_W_k_u_omega}) is immediate. The bound on
$\|u\|_{L_\infty(\Omega)}$ follows similarly.
\end{proof}

\subsection{Lipschitz domains in $\M$}

A domain $\Omega$ on a smooth, compact Riemannian manifold $\M$
satisfies an interior cone condition if there is a cone $C\subset
\RR^n$ with center $0$, aperture $2\varphi$, and radius $R$ such that,
with some orientation of $C$, $\Exp_p: C\to C_p \subset \Omega$. That
is, the image of the fixed cone $C$ is a geodesic cone $C_p$ contained
in $\Omega$. In addition, $\Omega$ satisfies the \emph{uniform cone
  condition} if, for every $p_0\in \partial \Omega$ and some
orientation of $C$, $\Exp_p(C\setminus \{0\}) \subseteq \Omega$ for all $p\in
\b(p_0,r)\cap \overline{\Omega}$. Finally, $\Omega$ is said to be
\emph{locally strongly Lipschitz} \cite{mitrea-taylor-99-1,
  hofmann-mitrea-taylor-07-1} if for every $p_0\in \partial \Omega$
there is a local chart $(U,\psi)$, $\psi:U\to \RR^n$,with
$\psi(p_0)=0$, a Lipschitz function $\lambda: \RR^{n-1}\to \RR$, with
$\lambda(0)=0$, and an $\varepsilon>0$ such that
\[
\psi(U\cap \Omega)=\{(x',\lambda(x')+t): 0<t<\varepsilon, x'\in
\RR^{n-1}, |x'|<\varepsilon\}.
\] 

Our approach to a manifold analogue of
Theorem~\ref{omega_euclidean_lip} is to employ a set of points for
$\M$ that are similar to those described in (\ref{defs_r_Tr}). The set
that we need is described and studied in \cite[\S
3]{grove-petersen-1988-1}. Let $\varepsilon>0$. There exists an
ordered set of points $\{p_1,\ldots,p_N\}\subset \M$ such that the
$\cup_{j=1}^N \b(p_j,\varepsilon) = \M$ and such that the balls
$\b(p_j,\varepsilon/2)$ are disjoint. Such a set is called a
\emph{minimal $\varepsilon$-net} in $\M$.%
\footnote{An $\varepsilon$-net is a set of points $X =\{p_1, \dots, p_N\}$  
for which $\bigcup \b(p_j, \varepsilon)$ covers $\M$ -- in other words, 
for which $h(X,\M) \le \varepsilon$. Likewise, a \emph{minimal $\varepsilon$-net} 
is quasiuniform, with separation distance $q\ge \epsilon/2$ and  mesh ratio $h/q\le2$.}
 It has the following two
important properties: First, there is a number $N_1=N_1(\varepsilon,
\M)$ for which $N\le N_1$. Second, there exists an integer
$N_2=N_2(\M)\ge 1$ such that for any $p\in\M$ the ball
$\b(p,\varepsilon)$ \emph{intersects at most} $N_2$ of the balls
$\b(p_j,\varepsilon)$. It is remarkable that $N_2$ is
\emph{independent} of $\varepsilon$ and, in fact, depends only on
general properties of $\M$ itself. We will need a slightly stronger
version of this result.

\begin{lemma}\label{eps_net_intersect_card}
  Let $\{p_1,\ldots,p_N\}$ be a minimal $\varepsilon$-set, $p\in \M$, and let
  $1\le \alpha$. Suppose $\varepsilon\le d_\M/\alpha$, where $d_\M$ is
  the diameter of $\M$. Then the cardinality $s:=\# \{p_j \colon
  \b(p,\alpha\varepsilon)\cap \b(p_j,\varepsilon)\ne \emptyset\}\le
  (4\alpha+1)^d e^{\frac{3(d-1)}{\sqrt{|\kappa|}}d_\M}$.
 \end{lemma}
\begin{proof}
  The argument used in \cite[Lemma 3.3]{grove-petersen-1988-1} gives,
  \emph{mutatis mutandis},
\[
s\le
\frac{\int_0^{(2\alpha+\frac12)\varepsilon}\sinh^{d-1}(\sqrt{|\kappa
    |\,}t)dt}{\int_0^{\varepsilon/2} \sinh^{d-1}(\sqrt{|\kappa\,
    |}t)dt} =:H (\alpha,\varepsilon,\kappa) =
H(\alpha,\varepsilon/\sqrt{|\kappa|},1)
\]
where $(d-1)\kappa$ is a lower bound on the Ricci curvature of
$\M$. Making use of $1\le \sinh(x)/x\le e^x$, we see that
\[
d^{-1}x^d \le \int_0^x t^{d-1}dt \le \int_0^x \sinh^{d-1}(t)dt \le
d^{-1}x^d e^{(d-1)x},
\]
and consequently that
\[
H(\alpha,\varepsilon,\kappa) \le (4\alpha+1)^d
e^{(d-1)(2\alpha+\frac12)\varepsilon/\sqrt{|\kappa|}}\le (4\alpha+1)^d
e^{\frac{3(d-1)}{\sqrt{|\kappa|}}d_\M},
\]
which completes the proof.
\end{proof}

\begin{lemma} \label{cone_p_j} 
Let $R<\inj/3$, $\varphi\in (0,\pi/2]$ and 
$\varepsilon =  \frac{\Gamma_1R \sin(\varphi)}{2(1+\sin(\varphi))}$. 
If $\{p_1, \dots, p_N\}$ is an $\varepsilon$-set
and if $C_p$ is a geodesic cone with
center $p$, radius $R$, and angle $\varphi$, then for some
$j$ we have that $\b(p_j,\varepsilon)\subset C_p$.
\end{lemma}
\begin{proof}
We will work in normal coordinates on $T_pM$, where the cone $C$ has  
vertex at the origin and $e_n=(0,\ldots,1)$ is chosen to be along the axis of $C$. 
The largest Euclidean ball in $C$ has radius  
$\rho=R\sin(\varphi)/(1+\sin(\varphi))$ 
and center $x_c=(R-\rho)e_n$. 
It follows that any ball having its center a Euclidean distance $\rho/2$ from $x_c$ 
and having its radius less than $\rho/2$ is also contained in $C$. 
Let $p_c=\Exp(x_c)$. 
Since the balls
$\b(p_j,\varepsilon)$, $j=1,\ldots,N$, cover $\M$, we can find $p_j$
such that $p_c\in \b(p_j,\varepsilon)$. 

Let $x_j=\Exp_p^{-1}(p_j)$. 
Equation (\ref{isometry}) implies that 
$|x_c-x_j|\le \d(p_c,p_j)/\Gamma_1<\varepsilon/\Gamma_1=\rho/2$. 
Now  consider the ball $\b(p_j,\Gamma_1\rho/2)$. 
Let $q\in \b(p_j,\Gamma_1\rho/2)$ and let $x=\Exp_p^{-1}(q)$.
Applying equation (\ref{isometry}) then yields that 
$|x-x_j|\le \d(p,q)/\Gamma_1<\rho/2$, and consequently that 
$\b(p_j,r)\subset  C_p$, with 
$r\le \rho/2=\frac{\Gamma_1R  \sin(\varphi)}{2(1+sin(\varphi))}$.
\end{proof}


Our goal is now to cover $\Omega$ with domains analogous to those used
in the previous section. To that end, let  $R\le R_\Omega$,  fix
\begin{equation}
r := \frac{\Gamma_1R \sin(\varphi)}{2(1+\sin(\varphi))} = 2 F(\varphi) R
\end{equation}
and find a minimal $\varepsilon$-net (with $\varepsilon = r$) $\{p_1,\dots,p_N\}$
and set
$T_r := \{p_j\colon \b(p_j,r)\subset \Omega\}$. 
Because $\Omega$ obeys a uniform cone
condition, with radius $R_\Omega$ and angle $\varphi$,
Lemma~\ref{cone_p_j} implies that $T_r$ is nonempty.

Next, for each $p_j\in T_r$, let $\stardom_j$ be the set of
all $p\in \Omega\cap \b(p_j,R)$ such that the geodesic convex hull of
$\{p\}\cup \b(p_j,r)$ -- i.e., the set comprising all points on every
geodesic connecting $p$ to a point in $\b(p_j,r)$ -- is contained in
$\Omega\cap \b(p_j,R)$. Again by Lemma~\ref{cone_p_j}, for every $p\in
\Omega$ there is a $p_j\in T_r$ such that the geodesic cone $C_p$
contains $\b(p_j,r)$. Since this cone also contains the geodesic convex
hull of $\{p\}\cup \b(p_j,r)$, it follows that $p\in \stardom_j$ and,
hence, that $\Omega =\cup_{p_j\in T_r} \stardom_j$.

We claim that the domain $D_j:=\Exp_{p_j}^{-1}(\stardom_j)$ is star
shaped with respect to the Euclidean ball
$\b(\Exp_{p_j}^{-1}(p_j),r/\Gamma_2)$. To show this, we will need the
following lemma.

\begin{lemma}\label{min_ray2pt_dist}
  Let $p\in \M$, $\bfu,\bfv \in T_{p}M$ satisfy $|\bfu |_{p} =
  |\bfv|_{p} =1$, $\alpha :=\arccos(\langle \bfu,\bfv\rangle) \in
  (0,\pi]$. If $p_\rho=\Exp_{p}(\rho\bfu)$, so that
  $\rho=\d(p,p_\rho)<\inj/3$, then the geodesic distance $r$ from
  $p_\rho$ to the ray along $\bfv$ satisfies
\[
\Gamma_1 \rho\sin \bigl(\min(\alpha,\pi/2)\bigr) \le r\le \Gamma_2 \rho\sin \bigl(\min(\alpha,\pi/2)\bigr).
\]
\end{lemma}

\begin{proof}
  Consider the sector in $\spam(\bfu,\bfv)$ formed by $t\bfu + s\bfv$,
  where $s, t\ge 0$. We will work in normal coordinates based at $p$. 
  The minimum geodesic distance $r$ from $p_\rho$ to geodesic $\Exp_p(s\bfv)$ occurs at a point $\Exp_p(t\bfv)$. In addition, the minimum Euclidean distance $r'$ from $\rho\bfu=\Exp_p^{-1}(p_\rho)$ to the ray will occur at another point, $t'\bfv$, where $\bfv$ is perpendicular to $t'\bfv-\rho\bfu$, in the Euclidean sense. These facts imply that $r=\d(p_\rho,t\bfv)\le
  \d(p_\rho,\Exp_p(t'\bfv))\le \Gamma_2 |\rho \bfu -
  t'\bfv|_{\text{\rm eucl}}$. Using a little trigonometry, together with $t'\bfv - \rho\bfu$ and
  $\bfv$ being perpendicular, we see that $|\rho\bfu - t'\bfv|_{\text{\rm
      eucl}}=\rho\sin(\alpha)$ when $\alpha<\pi/2$,  and  that 
      $|\rho\bfu - t'\bfv|_{\text{\rm eucl}}=|\rho\bfu |_{\text{\rm eucl}}=\rho$ when $\alpha\ge \pi/2$. In a similar way, we have
  $r=\d(p,\Exp_p(t\bfv))\ge \Gamma_1|\rho\bfu - t\bfv|_{\text{\rm  eucl}}\ge \Gamma_1 |\rho\bfu-t'\bfv|_{\text{\rm eucl}}=
  \Gamma_1\rho\sin\bigl(\min(\alpha,\pi/2)\bigr)$.  Combining the inequalities completes the
  proof.
\end{proof}

\noindent There is corollary to the lemma that will be useful for smooth surfaces, in particular balls and annuli. We state
and prove it now, although it will only become useful after the zeros result Theorem \ref{omega_manf_continuous_est}.

%
\begin{corollary}\label{ball_boundary}
Let $q\in \partial\Omega$ and suppose there is a ball $\b(p,\rho)$, $\rho<\inj/3$, such that $\b(p,\rho)\subset \Omega$ and that $d(q,p)=\rho$. Then, the geodesic cone $C_q$, with vertex $q$, axis along the geodesic joining $q$ to $p$, radius $\Gamma_1 \rho/2$ and angle $\varphi=\arcsin(\frac{1}{2\Gamma_2})$ satisfies $C_q\setminus \{q\}\subset \Omega$. 
\end{corollary}

\begin{proof}
Let $q'$ be a point on the lateral side of the cone that is $\rho$ away from $p$ --
denote it in coordinates around $q$ by $q'=s\bfv$. With this, we identify  two triangles.
\begin{figure}[ht]
\centering
 \psfrag{p}{$p$}
 \psfrag{p'}{$p'$}
 \psfrag{q}{$q$}
\psfrag{q'}{$q'$}
\psfrag{boundary}{$\partial \b(p,\rho)$}
  \includegraphics[height=2.5in]{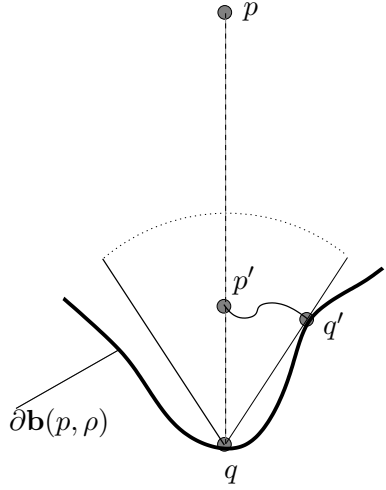}
  \caption{In this figure, $p$ is the center of the ball of radius $\rho$, $q$ is a point on the boundary,
  $q'$ is a point simultaneously on the boundary of the ball and on the side of the cone and 
  $p'$ is the nearest point on the ray $\Exp_q(t\bfv)$ to $q'$.}
  \end{figure}

The first triangle has corners $p =\Exp_q(\rho \bfu)$, $q'=\Exp_{q}(s\bfv)$ and a point on the ray
$\Exp_q(t \bfu)$ with $0<t<\rho$ ( 
 Lemma A.7 guarantees that the vertex at $p$ is acute, since $s\le \Gamma_1 \rho/2$). 
 Let us denote the third corner of this triangle by $p' = \Exp_q(t' \bfu)$
Note that $p'$ is a distance of $t'$ from $q$ and a distance of $\rho - t'$ from $p$.
The triangle inequality gives us that 
$$\rho\le \d(p',q')  + (\rho - t')  \quad \longrightarrow \quad t'\le\d(p',q'). $$
The second triangle we consider has corners $q$, $q'$ and $p'$, and Lemma A.7 ensures that
$\d(q',p') \le \Gamma_2 s\sin\alpha = s/2$. So the triangle inequality here gives us
that $s \le \d(p',q') + t'\le s/2 + t'$ -- so $s/2 \le t'$.

Combining  estimates from both triangles, we see that
$s/2 \le t' \le \d(p',q') \le s/2$, so both
$t'$ and $\d(p',q')$ are $s/2$.

In other words, the curve from $p$ to $q'$ has the same length as the curve from $p$ to
$p'$ to $q'$. By \cite[Corollary 3.9, p. 73]{docarmo_1992}, any piecewise differentiable 
curve joining two points -- $p$ to $p'$ to $q'$
our case -- with length less than or equal to any other such curve is a
geodesic. 
Because this occurs inside $\b(p, \inj)$, where geodesics do not
cross, there can only be one geodesic
joining $p$ and $q'$. Since $p$ to $p'$ has to be on the geodesic joining $p'$ to
$q'$, and since the length is $\rho$, $q'$ and $q$ coincide.
\end{proof}

\begin{proposition}\label{euclidean_star_D_j}
  The domain $D_j:=\Exp_{p_j}^{-1}(\stardom_j)$ is star shaped with
  respect to the Euclidean ball
  $B(\Exp_{p_j}^{-1}(p_j),\Gamma_1 r/\Gamma_2^2)$. Also, the chunkiness parameter
  and diameter for $D_j$ satisfy
\begin{equation}
\label{euclidean_bnds}
\gamma_{D_j} \le \frac{2\Gamma_2^2 R}{\Gamma_1 r} = 
\left(\frac{\Gamma_2}{\Gamma_1}\right)^2\frac{4(1+\sin(\varphi))}{\Gamma_1\sin(\varphi)}=
\frac{\Gamma_2^2}{\Gamma_1^2F(\varphi)} \ \text{and } d_{D_j}\le 2R.
\end{equation}

\end{proposition}

\begin{proof}
We begin by fixing a point $p\in \stardom_j$.
The geodesic convex hull of $\{p\}\cup \b(p_j,r)$ contains  a  largest cone
 with vertex $p$ and central axis the geodesic ray connecting
  $p$ to $p_j$. 
 On this cone, whose (lateral) surface consists of geodesics emanating from $p$, 
 there exists a geodesic  
 lying tangent to the sphere $\partial \b(p_j,r)$. In other
  words, we take the cone of largest aperture $2\alpha$ for which all  geodesics
  pass through $\b(p_j,r)$. 
\begin{figure}[ht]
\centering
 \psfrag{r}{$r$}
 \psfrag{pr}{$p_{\rho}$}
 \psfrag{p}{$p$}
\psfrag{pj}{$p_j$}
\psfrag{R}{$r'$}
  \includegraphics{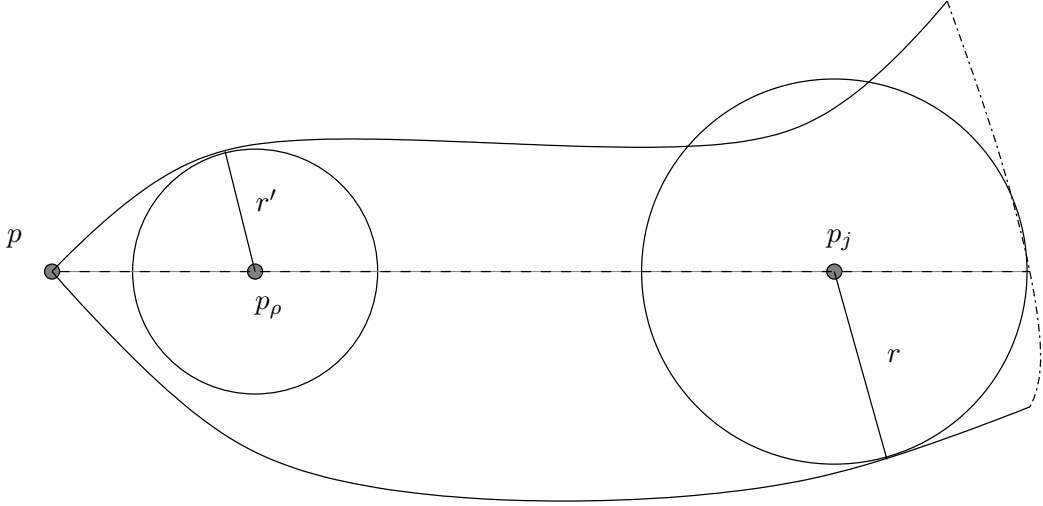}
  \caption{The largest cone with vertex $p$ and with central axis the geodesic that connects $p$ to $p_j$  
  for which no geodesic lies outside of the ball $\b(p_j, r)$. 
The radius $r'$ of the ball centered at $p_{\rho}$ lying tangent to the cone is greater than  
$\frac{\Gamma_1 r\rho}{\Gamma_2\rho_j}$}
  \end{figure}
 
  From  this two things follow. 
  First, by Lemma~\ref{min_ray2pt_dist}, the distance from $p$
  to $p_j$, $\rho_j=\d(p,p_j)$, the  angle $\alpha$, 
  and the radius $r$ are related by   $r\le \Gamma_2 \rho_j \sin\bigl(\min(\alpha,\pi/2)\bigr)$. 
  Second, for $\rho \le \rho_j$ and for a point $p_{\rho}$ lying a distance of $\rho$ from $p$ along the central axis, the
  distance $r'$ of $p_\rho$ to the surface of the cone satisfies: 
  $$r'\ge \Gamma_1 \rho\sin\bigl(\min(\alpha,\pi/2)\bigr) \ge \frac{\Gamma_1 r \rho}{\Gamma_2 \rho_j}.$$ 
  It follows the cone contains the ball
  $\b(p_\rho,(\Gamma_2 \rho_j) ^{-1}\Gamma_1r \rho)$, which obviously is also contained in the convex hull of 
  $\{p\}\cup \b(p_j,r)$. 
  
Shifting to normal coordinates centered at $p_j$, rather than at $p$, we see that the geodesic ball
$\b(p_\rho,(\Gamma_2 \rho_j) ^{-1}\Gamma_1r \rho)$ 
contains the (image of the) Euclidean ball
$B(\Exp_{p_j}^{-1}(p_\rho),(\Gamma_2^2 \rho_j)^{-1}\Gamma_1r\rho)$. 
Indeed, if $|y-\Exp_{p_j}^{-1}(p_\rho)|\le (\Gamma_2^2 \rho_j)^{-1}\Gamma_1r\rho$   then
$$\d(\Exp_{p_j}(y), p_{\rho}) \le \Gamma_2 |y -  \Exp_{p_j}^{-1}(p_\rho)| \le \frac{\Gamma_1r \rho}{\Gamma_2 \rho_j}.$$ 
A straightforward argument using Euclidean geometry implies that the
  Euclidean convex hull of 
  $\{\Exp_{p_j}^{-1}(p)\}\cup B(0,\Gamma_1r/\Gamma_2^2)$ 
  is contained in $D_j$. Hence, in the Euclidean
  metric, $D_j$ is star shaped with respect to the ball
  $B(\Exp_{p_j}^{-1}(p_j), \Gamma_1r/\Gamma_2^2)$. Moreover, since
  $\stardom_j\subset \b(p_j,R)$, we have that $D_j\subset
  B(\Exp_{p_j}^{-1}(p_j), R )$.  Finally, from these facts it
  is easy to see that the bounds in (\ref{euclidean_bnds}) hold.

\end{proof}

Applying this together with Proposition~\ref{stardom_est}  and  Lemma~\ref{Fran} , we have the following result.

\begin{proposition}\label{manif_stardom_est}
  Let $\stardom_j$ be as above,
  $m\in \nats$ and $\sfp\in \reals$, $1\le \sfp \le \infty$. Assume
  $m>d/\sfp$ when $\sfp>1$, and $m\ge d$, for $\sfp=1$. If $u\in
  W_\sfp^m(\stardom_j)$ satisfies $u|_X=0$, where
  $X$ is a finite subset of $\stardom_j$, and if the geodesic meshnorm $h=h_X\le \frac{\Gamma_1^5 F(\varphi)^2}{8m^2 \Gamma_2^4}R$, then
\begin{equation}
\label{p_bound_W_k_u_M}
\|u\|_{W_\sfp^k(\stardom_j)} \le C_{m,k,\sfp,\M} R^{m-k}F(\varphi)^{-d-2k}
\|u\|_{W_\sfp^m(\stardom_j)}
\end{equation}
\begin{equation}
\label{infty_bnd_u_M}
\|u\|_{L_\infty(\stardom_j)} \le C_{m,k,\sfp,\M} R^{m-d/\sfp} F(\varphi)^{-d} 
\|u\|_{W_\sfp^{m}(\stardom_j)}.
\end{equation}
\end{proposition} 

\begin{proof}
We first apply Proposition~\ref{stardom_est} with $\tilde h$ 
the Euclidean meshnorm for $\Exp_{p_j}^{-1} X$, satisfying 
$\tilde h\le h/\Gamma_1
\le \frac{d_{D_j}}{16m^2\gamma_{D_j}^2}\le \frac{2\Gamma_1^4 RF(\varphi)^2}{16m^2 \Gamma_2^4}$. 
Then, equations (\ref{p_bound_W_k_u})  and (\ref{infty_bnd_u}) hold for $u\circ \Exp_{p_j}$ on $D_j$, with $\gamma_{D_j}$ and $d_{D_j}$ replaced by the bounds in (\ref{euclidean_bnds}). Applying Lemma~\ref{Fran} then gives the bounds above. 
\end{proof}

\begin{theorem}[{\bf Manifold Case}]
  \label{omega_manf_continuous_est}
  Suppose that $\Omega\subseteq \M$ is a bounded, Lipschitz domain that satisfies a uniform cone condition, with the cone having radius $R_\Omega<\inj/3$ and angle $\varphi$. Let $k$, $m$, and $\sfp$ be as in
  Proposition~\ref{stardom_est}, and let $X\subset \Omega $ be a
  discrete set with mesh norm $h$ satisfying
\begin{equation}
\label{h_cond_omega}
h\le  h_0 R_\Omega, \ \ h_0:=\frac{\Gamma_1^5}{8 m^2\Gamma_2^4}  F(\varphi)^2, 
\end{equation}
where $\Gamma_1, \Gamma_2$ and $F(\cdot)$ are defined in (\ref{isometry}) and
(\ref{defs_r_Tr}), respectively.  If $u\in W_\sfp^m(\Omega)$ satisfies
$u|_X=0$, then we have
\begin{equation}
\label{p2p_omega}
\|u\|_{W_\sfp^k(\Omega)} \le C_{m,k,\sfp,\M}F(\varphi)^{-(1+1/\sfp)d-2m}h^{m-k} \|u\|_{W_\sfp^m(\Omega)}
\end{equation} 
and
\begin{equation}
\label{inf2p_omega}
\|u\|_{L_\infty(\Omega)} \le C_{m,\sfp,\M} h^{m-d/\sfp}F(\varphi)^{-d+2d/\sfp-2m} 
\|u\|_{W_\sfp^{m}(\Omega)}.
\end{equation}
\end{theorem}  

\begin{proof}
We are given $h$ in (\ref{h_cond_omega}) to begin with. Thus, we may choose $R=\frac{8 m^2\Gamma_2^4} {\Gamma_1^5} hF(\varphi)^{-2}\le R_\Omega$, and also take the $\stardom_j$'s to be the domains corresponding to this $R$. It follows that the conditions on $h$ in Proposition~\ref{manif_stardom_est} hold; consequently, 
\begin{eqnarray}
\|u\|_{W_\sfp^k(\stardom_j)} &\le& C_{m,k,\sfp,\M} h^{m-k}F(\varphi)^{-d-2m}\|u\|_{W_\sfp^m(\stardom_j)} 
\label{p2p_star},\\
\|u\|_{L_\infty(\stardom_j)} &\le& C_{m,\sfp,\M} h^{m-d/\sfp}F(\varphi)^{-d+2d/\sfp-2m} 
\|u\|_{W_\sfp^{m}(\stardom_j)}. \label{inf2p_star}
\end{eqnarray}
Because of the decomposition $\Omega=\cup_{p_j\in T_r}\stardom_j$, the bound in (\ref{inf2p_star}) immediately implies (\ref{inf2p_omega}). Moreover, this decomposition also gives us
\[
\|u\|_{W_\sfp^k(\Omega)}^{\sfp} \le \sum_j\|u\|_{W_\sfp^k(\stardom_j)}^{\sfp} \le  (C_{m,k,\sfp,\M} h^{m-k}F(\varphi)^{-d-2m})^{\sfp}(\sum_j \|u\|_{W_\sfp^m(\stardom_j)}^{\sfp}).
\]
From Definition~\ref{Sob_Norm}, we see that
\[
\sum_j \|u\|_{W_\sfp^m(\stardom_j)}^{\sfp} = \sum_{i=0}^m
    \int_{\Omega}\sum_j  \chi_{\stardom_j}(p)
    | \nabla^i f |_{g,p}^\sfp
    \, \dif \mu(p) \le \sup_{p\in \Omega} \left(\sum_j \chi_{\stardom_j}(p)\right) \|u\|_{W_\sfp^m(\Omega)}^p.
\]
The sum $\sum_j \chi_{\stardom_j}(p)$ is precisely the number of $\stardom_j$'s that  contain $p$. Since $\stardom_j \subset 
\b(p_j,R)$, $p_j$ is itself also in $\b(p,R)$.  Consequently, the number of $\stardom_j$'s containing $p$ is bounded above by the number of balls $\b(p_j, r)$, where $r=2RF(\varphi)\sim h/F(\varphi)$, that contain $p$, and, ultimately by the maximum number of $\b(p_j,r)$'s that can intersect each $\b(p,R)$.  By Lemma~\ref{eps_net_intersect_card}, this is $(4\alpha+1)^de^{\frac{3(d-1)}{\sqrt{|\kappa|}}d_\M}$, where $\alpha = R/r=\frac12 F(\varphi)^{-1}$. 
Putting together the two previous inequalities then yields
\[
\|u\|_{W_\sfp^k(\Omega)}^{\sfp} 
\le 
(C_{m,k,\sfp,\M} h^{m-k}F(\varphi)^{-d-2m})^{\sfp}\,2^{2d}F(\varphi)^{-d} e^{\frac{3(d-1)}{\sqrt{|\kappa|}}d_\M} \|u\|_{W_\sfp^m(\Omega)}^{\sfp}.
\]
Taking the $\sfp^{th}$ root, lumping constants, and manipulating the result, we obtain (\ref{p2p_omega}).
\end{proof}

We remark that the various constants appearing in Theorem \ref{omega_manf_continuous_est}, including $h_0$, only depend on $\varphi$, and only the right side of (\ref{h_cond_omega})  depends on the radius $R_\Omega$, and that dependence is linear. Thus, the dependence on $\Omega$ is completely explicit.

At this point we can extend the Duchon type error estimates for approximation by conditionally positive definite 
kernels. To our knowledge, this is the first result of this kind on bounded regions in compact Riemannian
manifolds.

To this end, suppose $K:\Omega\times\Omega \to \reals$ is positive definite (i.e. the matrix 
$\bigl(K(\xi,\zeta)\bigr)_{\xi,\zeta}$ is positive definite for each $\Xi$ -- see Definition \ref{cpd}) and consider that the ``native space'' $\mathcal{N}_K$, the reproducing kernel Hilbert space
constructed by taking the space of  arbitrary linear combinations of $K(\cdot,\xi)$, 
completed under the inner product $\langle f, g\rangle \mapsto \sum_{\xi,\zeta} A_{\xi}B_{\zeta}K(\xi,\zeta)$ for 
$f = \sum A_{\xi} K(\cdot,\xi)$ and  $g = \sum B_{\zeta}K(\cdot,\zeta)$. In this case, it is well known that the kernel interpolant $I_{\Xi} f$ is the optimal interpolant in the sense of $\mathcal{N}_K$. Namely, $\|I_{\Xi}f\|_{\mathcal{N}_K} \le\|s\|_{\mathcal{N}_K}$ for all $s\in\mathcal{N}_K$ with $s|_{\Xi} = f|_{\Xi}$.

\begin{corollary} Let $m>d/2$, and 
let $K$ be a positive definite kernel on $\M$
for which $\mathcal{N}_K$ is continuously embedded in $W_2^m(\M)$.
Let $\Omega\subset \M$ satisfy a uniform cone condition with radius $R_{\Omega}\le \inj/3$ and 
angle $\varphi$. 
For $\Xi \subset \Omega$ having meshnorm $h\le h_0 R_{\Omega}$ and 
for  $f\in \mathcal{N}_K$,
$$\| f - I_{\Xi}f \|_{L_\infty(\Omega)} \le 
C_{K} F(\varphi)^{-2m} 
h^{m-d/2}\|f  \|_{\mathcal{N}_K(\M)} .$$
\end{corollary}
We note that this result holds for a much larger class of kernels than considered 
in the previous sections (i.e., defined by Definition \ref{polyharmonic_kernels}).
In particular, there are numerous examples of compactly supported 
kernels on $\reals^d$ and $\sphere$ having native spaces that are Sobolev spaces,
but which do not invert an elliptic differential operator. 
However, his type of error estimate should be 
compared to those in Corollary \ref{sphere_result} and Corollary \ref{SO(3)_result} -- 
observe that the condition on the target function is quite restrictive (it needs to be in
$\mathcal{N}_K$ and there is a basic disagreement between the approximation order 
$m-d/2$ and the smoothness assumption). 
\begin{proof}
Apply Theorem \ref{omega_manf_continuous_est} to  $u =f-I_{\Xi}$, we see that
\begin{eqnarray*}
\| f - I_{\Xi}f \|_{L_\infty(\Omega)}^2 &\le&
 \left(C_{m,\M} h^{m-d/\sfp}F(\varphi)^{-2m} \right)^2 \| f - I_{\Xi}f \|_{W_2^m(\Omega)}^2\\
&\le& \left(C_{m,\M} h^{m-d/\sfp}F(\varphi)^{-2m} \right)^2 \| f - I_{\Xi}f \|_{W_2^m(\M)}^2\\
&\le& C \left(C_{m,\M} h^{m-d/\sfp}F(\varphi)^{-2m} \right)^2 \| f - I_{\Xi}f \|_{\mathcal{N}_K}^2\\
&\le&C_{K} F(\varphi)^{-4m} 
h^{2m-d}\|f  \|_{\mathcal{N}_K}^2 .
\end{eqnarray*}
The next to last inequality is the embedding $\mathcal{N}_K \subset W_2^m(\M)$,
while the last inequality is the Pythagorean theorem for orthogonal projectors
$\| f - I_{\Xi}f \|_{\mathcal{N}_K}^2 + \|  I_{\Xi}f \|_{\mathcal{N}_K}^2= \| f  \|_{\mathcal{N}_K}^2$.
\end{proof}

There are several domains that are important for us, and that we will discuss below. We begin with the manifold itself. In that case, we may take $R_\Omega = \inj/3$. The angle $\varphi$ may be set equal to $\pi/2$, because every such cone is contained in $\M$. This means that $F(\varphi)=F(\pi/2) = 1/8$. 
\begin{corollary}[{\bf Full Manifold}]
  \label{omega_full_manf_continuous_est}
  Suppose that  $\M$ is compact. Let $k$, $m$, and $\sfp$
  be as in Theorem \ref{omega_manf_continuous_est}. Then, with $h_0:=\frac{\Gamma_2^2}{8^3 m^2\Gamma_1^2} $, there is  a constant $C_{m,k,d,\M}$ such that if $X\subset \M$ 
  has  mesh norm $h\le h_0\inj/3$ 
  and if $u\in W_\sfp^m(\M)$ satisfies $u|_X=0$,
then
\begin{equation}
\label{p_bound_W_k_u2}
\|u\|_{W_\sfp^k(\M)} \le 
 C_{m,k,\sfp,\M}h^{m-k}\|u\|_{W_\sfp^m(\M)} .
\end{equation} 
\end{corollary}

The domains that we now turn to are balls, annuli, and complements of balls. 
In all of the cases discussed below, the domains $\Omega$ satisfy the ball property described in Corollary~\ref{ball_boundary} at each point $q\in \partial \Omega$. 
Consequently, we may take 
$\varphi=\arcsin(\frac{1}{2\Gamma_2})$, and so 
$F(\varphi)=1/(8\Gamma_2+4)$. It thus follows that in all such cases 
\begin{equation}
\label{common_h_0}
h_0 =
\frac{\Gamma_1^5}{ 128m^2\Gamma_2^4(2\Gamma_2+1)^2}.
\end{equation}
and all of the factors in Theorem \ref{omega_manf_continuous_est} depend only on parameters from the manifold itself, as well as $p,k,m$, but not at all on the center and the radius of the ball/annulus/punctured ball.

We now turn to balls. A ball of any size may be treated; however, if the the radius is larger that $\inj$ it may intersect itself, giving rise to corners with angles that have to be dealt with on a case by case basis. This phenomenon is easy to see in the case of the torus embedded in $\RR^3$, where a sufficiently large ball begins to wrap back on itself.   When the radius of the ball is less than $\inj$, this wrapping doesn't happen. With this assumption, we have the following result:

\begin{corollary}[{\bf Zeros estimate on balls}] \label{zl_balls}
  Assume $m>d/2$.  Suppose that $r<\inj$. 
  If $u\in W_p^m(\M)$ vanishes on $X\subset \b(p,r)$, where $h\le h_0r/2$, where $h_0$ is given by (\ref{common_h_0}), we have
 \[
\|u\|_{W_p^k(\b(p,r))}\le C_{m,k,\sfp,\M} h^{m-k} |u |_{W_p^m(\b(p,r))}.
\]
\end{corollary}
\begin{proof}
Obviously every point $q\in \partial \b(p,r)$ satisfies the conditions in Corollary~\ref{ball_boundary}. A direct application of Theorem~\ref{omega_manf_continuous_est} then completes the proof.
\end{proof}
\vspace{3pt}
\begin{corollary}[{\bf H{\" o}lder estimate on balls}] \label{he_balls} If $m$ is greater than $d/2+\epsilon$, and the conditions
of Corollary~\ref{zl_balls}hold (in particular, 
 $r$ is less than $\inj$, and 
 $h \le  h_0 r$), and if 
$u \in W^{m}_2 (\b(p,r))$ satisfies $u{|_{X}} = 0$, then for every $z\in \b(p,r)$,
$$|u(p)-u(z)| \le C r^{m-\epsilon-d/2}\d(p,z)^{\epsilon}\|u\|_{W^{m}_2(\b(p,r))},$$ 
where $C$ is a constant depending only on $m$, $\M$ and $\epsilon$. 
\end{corollary}
\begin{proof}
This follows because  the Sobolev embedding theorem, in conjunction with Lemma \ref{Fran} ensures that $\tilde{w} = w\circ \Exp_p\in C^{\epsilon}(B(0,\inj))$. Thus for $z = \Exp_p(x)\in \b(p,\inj)$,
$$\frac{|w(p)-w(z)|}{\d(p,z)^{\epsilon}} = \frac{|\tilde{w}(0)-\tilde{w}(x)|}{|x|^{\epsilon}}\le
|\tilde{w}|_{C^{\epsilon}(B(0,\inj))}\le C\|\tilde{w}\|_{W_2^m(B(0,\inj))}.$$
For a general $r<\inj$, set $\tilde{w}(\frac{\inj}{r} x) = \tilde{u}(x)$. 
Then 
$$\frac{|u(p)-u(z)|}{\d(p,z)^{\epsilon}}
\le \left(\frac{\inj}{r}\right)^{\epsilon}C\|\tilde{w}\|_{W_2^m(B(0,\inj))}
=\left(\frac{\inj}{r}\right)^{\epsilon}C\left(\sum_{k\le m}\left(\frac{r}{\inj}\right)^{2k} |\tilde{u}|_{W_2^k(B(0,r))}^2\right)^{1/2}.
$$
The result follows by applying Lemma \ref{Fran} in conjunction with Corollary \ref{zl_balls}.
\end{proof}
\noindent A similar argument to the proof of Corollary \ref{zl_balls} given above yields these results for annuli and complements of balls. In the  following two lemmas, we are concerned with the case $p=2$. Consequently, we suppress dependence on these parameters by expressing the constant from the zeros lemma for such domains simply as $\Lambda$.
In other words,
\begin{equation}\label{zeros_const}
\Lambda := \max_{k=0\dots m-1}C_{m,k,2,\M}\left(8\Gamma_2+4\right)^{(3/2)d+2m},
\end{equation}
which depends only on $m$ and $\M$.
%
%
\begin{corollary}[{\bf Zeros lemma on annuli}] \label{zl_annulus}
  Assume $m>d/2$, Let $\a  = \b(p,r) \setminus \b(p,r-t)$, where $0<t<r<\inj$, and let $h_0$ be given by (\ref{common_h_0}). 
  If $u\in W_p^m(\a)$ vanishes on $X\subset \a$, where $h\le h_0\min(t/2,\inj/3)$, we have
 \[
\|u\|_{W_2^{k}(\a)}\le \Lambda h^{m-k} |u |_{W_p^m(\a)}.
\]
\end{corollary}
\begin{proof}
At each point $q$ of the boundary of $\a$ an open ball of radius $t/2$ can be placed inside $\a$
with a boundary that passes through $q$. The result follows from Corollary \ref{ball_boundary} and Theorem \ref{omega_manf_continuous_est}.
\end{proof}
\begin{corollary}[{\bf Zeros lemma on complements of balls}] \label{zl_com}
If $r<\inj/3$  and 
if $u\in W_2^m(\M)$ vanishes on $X$ with $h = h(X,\b(p,r)^{c} ) \le h_0\inj/3$, then 
$$\|u\|_{W_2^{m-k}(\b(p,r)^{c})}\le\Lambda  h^{m-k} 
|u |_{W_2^m(\b(p,r)^{c})}.$$
\end{corollary}
\begin{proof}
By placing its center, $q$, a distance of  
$r+2 \inj/3$ 
away from $p$, the ball $\b(q,2\inj/3)$ can be placed in  $\b(p,r)^{c}$. It follows that 
the set satisfies a cone condition with radius $R = \inj/3$.
\end{proof}

\end{appendix}


\bibliographystyle{siam}
\bibliography{HNW-Polyharmonic}

\end{document}